\documentclass[12pt]{article}
\usepackage{graphicx} % Required for inserting images
\RequirePackage{amsthm,amsmath,amsfonts,amssymb}
\RequirePackage[numbers]{natbib}
\RequirePackage[colorlinks,citecolor=blue,urlcolor=blue]{hyperref}
\usepackage[utf8]{inputenc}%(only for the pdftex engine)
\usepackage{mathtools} % amsmath with fixes and additions
\usepackage{booktabs} % commands to create good-looking tables
\usepackage{tikz, pgfplots} % nice language for creating drawings and diagrams
\usepackage[a4paper,width=150mm,top=25mm,bottom=25mm]{geometry}
\usepackage{amsmath}
\usepackage{amssymb}
\usepackage{mathtools}
\usepackage{xcolor}
\usepackage{caption}
\usepackage{subcaption}
\usepackage{verbatim}
\usepackage{titlesec}
\usepackage{csquotes}
\usepackage{authblk}

\newtheorem{theorem}{Theorem}
\newtheorem{proposition}{Proposition}
\newtheorem{lemma}{Lemma}
\newtheorem{corollary}{Corollary}

\theoremstyle{definition}
\newtheorem{definition}{Definition}
\newtheorem{example}{Example}
\newtheorem{remark}{Remark}
\newtheorem{assumption}{Assumption}
\newtheorem{question}{Question}

\newcommand{\X}{\mathbb{X}}
\newcommand{\C}{\mathbb{C}}
\newcommand{\Z}{\mathbb{Z}}
\newcommand{\R}{\mathbb{R}}
\newcommand{\E}{\mathbb{E}}
\newcommand{\F}{\mathcal{F}}
\newcommand{\G}{\mathcal{G}}
\newcommand{\D}{\mathcal{D}}
\newcommand{\T}{\mathcal{T}}

\newcommand{\Fsep}{H}
\newcommand{\LFsep}{L}
\newcommand{\fsep}{h}
\newcommand{\intsep}{f}
\newcommand{\lfsep}{l}
\newcommand{\sdint}{\mathcal{S}^\internal}
\newcommand{\sdlint}{\mathcal{S}^\linternal}

\newcommand{\sd}{\mathcal{S}}

\newcommand{\sgn}{\mathrm{sgn}}
\newcommand{\ttop}{\mathrm{top}}
\newcommand{\htr}{\mathrm{htr}}
\newcommand{\internal}{\mathrm{I}}
\newcommand{\linternal}{\mathrm{LI}}
\newcommand{\Left}{\mathrm{Left}}
\newcommand{\Right}{\mathrm{Right}}
\newcommand{\treks}{\mathbf{T}}
\newcommand{\paths}{\mathbf{P}}
\newcommand{\stable}{\mathrm{stb}}
\newcommand{\full}{\mathrm{full}}
\newcommand{\diag}{\mathrm{diag}}
\newcommand{\res}{\mathrm{res}}
\newcommand{\pd}{\mathrm{PD}}
\newcommand{\rank}{\mathrm{rank}}
\newcommand{\pa}{\mathrm{pa}}

\titleformat{\title}{display}{}{}{}[]
\titleformat{\section}[runin]{\normalsize\bfseries}{\thesection.}{0.5em}{}[.]
\titleformat{\subsection}[runin]{\normalsize\itshape}{\thesubsection.}{0.3em}{}[.]

\providecommand{\keywords}[1]
{
  \small	
  \textbf{\textit{Keywords---}} #1
}
\title{\large\textbf{CAUSAL INFERENCE ON PROCESS GRAPHS \\ PART II} \\ \normalsize{\textbf{CAUSAL STRUCTURE AND EFFECT IDENTIFICATION}}}
\author[1]{\normalsize Nicolas-Domenic Reiter \thanks{nicolas-domenic.reiter@dlr.de}}
\author[2,1]{ \normalsize Jonas Wahl \thanks{wahl@tu-berlin.de}}
\author[1]{ \normalsize Andreas Gerhardus \thanks{andreas.gerhardus@dlr.de}}
\author[1,2,3]{\normalsize Jakob Runge \thanks{jakob.runge@tu-dresden.de}}
\affil[1]{\normalsize German Aerospace Center (DLR), Institute of Data Science, Jena, Germany}
\affil[2]{\normalsize Technische Universität Berlin, Berlin, Germany}
\affil[3]{\normalsize Center for Scalable Data Analytics and Artificial Intelligence (ScaDS.AI) Dresden/Leipzig, TU Dresden, Germany}
\date{}
\begin{document}

\maketitle
\begin{abstract}
 A structural vector autoregressive (SVAR) process is a linear causal model for variables that evolve over a discrete set of time points and between which there may be lagged and instantaneous effects. The qualitative causal structure of an SVAR process can be represented by its finite and directed process graph, in which a directed link connects two processes whenever there is a lagged or instantaneous effect between them. At the process graph level, the causal structure of SVAR processes is compactly parameterised in the frequency domain. In this paper, we consider the problem of causal discovery and causal effect estimation from the spectral density, the frequency domain analogue of the auto covariance, of the SVAR process. Causal discovery concerns the recovery of the process graph and causal effect estimation concerns the identification and estimation of causal effects in the frequency domain.
 We show that information about the process graph, in terms of $d$- and $t$-separation statements, can be identified by verifying algebraic constraints on the spectral density. Furthermore, we introduce a notion of rational identifiability for frequency causal effects that may be confounded by exogenous latent processes, and show that the recent graphical latent factor half-trek criterion can be used on the process graph to assess whether a given (confounded) effect can be identified by rational operations on the entries of the spectral density.  
\end{abstract}
\keywords{Causal inference; SVAR processees; spectral density; identifibality}
\section{Introduction}
Questions such as \textquote{what are the causes of a given phenomenon?} or \textquote{how strong is the effect of one phenomenon on another?} are asked and studied in fields as diverse as climate science \cite{runge2023causal, MemoryMattersACaseforGrangerCausalityinClimateVariabilityStudies, ebert2012causal}, social science and economics \cite{bernanke1995blackbox, imbens2024causal}, and neuroscience \cite{Seth3293, siddiqi2022causal}. Causal inference \citep{pearl2009causality, peters2017elements} formalises these questions mathematically so that they can be analysed at an abstract level. 

Causal models are the mathematical object with which to formalise causal questions. It consists of a (causal) graph and a compatible stochastic functional model, commonly called \emph{Structural Causal Model (SCM)} \cite{pearl2009causality, peters2017elements}. The vertices on the causal graph represent the entities among which causal relationships are to be modeled and a directed edge between two vertices indicates an immediate cause-effect relation between the corresponding entities. The functional model assigns a stochastic object to each vertex on the causal graph, along with a set of functions that specify how each stochastic object is affected by its causes. One goal in causal inference is to understand the conditions on the causal model under which certain aspects of the model can be reconstructed from observational data. This goal splits into two main research directions of which the first is known as causal structure identification and the second as causal effect identification. Causal structure identification concerns the reconstruction of the causal graph from observational data, whereas causal effect identification is about extracting information about the functional model from observations, given at least partial knowledge of the causal graph. 

Linear Gaussian \emph{structural causal models (SCM)} \cite{bollen1989structural} are an important and extensively researched class of causal models. A linear Gaussian SCM assigns a Gaussian noise term to every vertex and a coefficient to every edge on the causal graph. The observational distribution of a linear Gaussian SCM model is a multivariate Gaussian distribution. Given that a Gaussian distribution is uniquely characterised by its mean and covariance matrix, the covariance represents the observational information from which one seeks to recover the causal model. Since the entries in the covariance are rational functions of the parameters of the linear model, this problem is effectively approached from an algebraic point of view \cite{drton2018algebraic, drton2009lectures}. From this point of view, causal structure identification is about expressing information of the causal graph in terms of algebraic relations between the entries in the covariance matrix. Results in this fashion include the algebraic characterisation of $d$- and $t$-separation in the underlying causal graph \cite{pearl2009causality, pearl1990independence, peters2017elements} resp. \cite{sullivant2010trek}.
Similarly, causal effect identification in linear Gaussian SCM's specialises to the problem of recovering the parameters of the causal model from the covariance matrix. The identification of causal effects is an interesting problem when some of the variables are hidden, i.e. not directly included in the covariance matrix. Algebraically, this means that the causal model for the observed variables must be recovered from a submatrix of the covariance. The structure of the causal graph (including the hidden variables) imposes algebraic relations between the entries in the covariance, and these relations allow or prevent certain aspects of the observable causal model to be recovered. In this regard, several graphical criteria have been identified for deciding whether, in a given causal graph, a specific causal effect can be identified from the covariance for almost all linear Gaussian SCMs compatible with the graph \cite{drton2011global, 10.1214/22-AOS2221, foygel2012half, weihs2018determinantal}. In order to apply these criteria to identify causal effects from observational data it is necessary that the observations are identically distributed and independent of each other.

However, in many practical applications, observations are time-dependent \cite{runge2023causal, FRISTON2013172, runge2019inferring, runge2019detecting, MemoryMattersACaseforGrangerCausalityinClimateVariabilityStudies}. In such a situation it is necessary to have a causal model that takes time into account. A simplified way of incorporating temporal dependencies into the causal model is to assume that the system being modelled evolves over a discrete set of time points. The causal structure in such systems can be represented graphically by drawing a vertex for each process at every time point, so that a directed edge between vertices expresses a delayed or simultaneous effect. This construction yields an infinite causal graph. Graphs of this particular type are well established in the literature \cite{dahlhaus2003causality, peters2017elements, gerhardus2024characterization, runge2019inferring} and are, among other names, often referred to as time series graphs. The infinite time series graph is reduced to a finite graph by projecting away the time coordinate \cite{peters2017elements, runge2019inferring, assad2022timeseries}. In this paper, we call this reduced graph the process graph of the time series graph. 

Structural vector autoregressive (SVAR) processes are theoretically well understood \cite{lutkepohl2005new, brockwell2009time} and widely used across scientific disciplines to model dynamically interacting quantities \cite{FRISTON2013172, Seth3293}. SVAR processes can be seen as a linear parameterisation of an appropriate causal time series graph. This leads to the question of how to systematically formalise the questions of causal structure identification and causal effect identification for linear causal time series models. The effectiveness and clarity of the algebraic perspective on causal inference for linear causal models with iid observations motivates the search for a similar formulation for linear time series models. 
The time series graph, together with its linear SVAR parameterisation, appears to be the natural choice for a causal model to be identified from observations \cite{gerhardus2024characterization, runge2019inferring, runge2019detecting, runge2020discovering, peters2017elements, gerhardus2020high, thams2022instrumental, mogensen2022instrumental, gerhardus2023projecting}. Using or making assumptions about the time series graph can be useful in recovering causal information from observations, e.g. cyclic relationships between processes could be resolved. However, for a given number of processes, the infinite time series graph can be arbitrarily complex. The corresponding process graph, on the contrast, is a finite graph, so its complexity is limited by the number of processes to be modelled. From a combinatorial point of view, it would therefore be attractive to use only the process graph instead of the time series graph when analysing the causal structure \cite{granger1969investigating, geweke1982measurement, geweke1984measures, eichler2010granger, eichler2010graphical}. 

In the first part of this two-part paper series \cite{reiter2023formalising}, we established that SVAR processes can be represented as linear causal models of stochastic processes on the finite process graph \cite{reiter2023formalising}, and we termed this reformulation the \emph{structural equation process (SEP)} representation. In this process representation, each vertex is associated with a stochastic process called the internal dynamics, which is similar to the Gaussian noise term in a classical SCM, and each edge on the process graph is associated with a filter. This linear causal model of stochastic processes can be equivalently represented in the frequency domain, where causal effects can be computed by means of a generalised path rule. The SEP representation is also relevant to applications because some questions are more naturally formulated in the frequency domain. In another related work \cite{reiter2024asymptotic}, we analyse statistical properties of causal effects in the frequency domain and illustrate the developed methods with a an example from the climate sciences, the effect of the 10-11 year solar cycle on the North Atlantic Oscillation (NAO). 

In this paper, we use the frequency representation of an SVAR process to give an algebraic formulation of the causal structure and causal effect identification problem in the frequency domain. The structure of this paper is as follows. In Section \ref{sec: preliminaries}, we recall the necessary concepts on SVAR processes as causal models for dynamically interacting variables. The object from which we seek to identify the frequency representation of the causal model, i.e. the process graph together with the Fourier transformed link filters, is the spectral density or spectrum of the process. In Section \ref{sec: frequnency domain parameterisation}, for a fixed time series graph, we parameterise the spectrum of a SVAR process as a matrix over the field of rational functions with real coefficients. Based on this parameterization we develop the main results of this paper.

\textsc{Causal structure identification}. Section \ref{sec: structure identifiability} concerns the (partial) identification of the process graph. Our first result concerns $d$-separation on the process graph. The following is an informal version of this result.   
\begin{theorem}[$d$-separation of processes]
    Let $G=(V,D)$ be an acyclic process graph, and let $X,Y,Z \subset V$ be three disjoint sets of process indices. For a generic SVAR process $\X$ compatible with $G$, it holds that the subprocesses $\X_X$ and $\X_Y$ are conditionally independent given $\X_Z$ if and only if $X$ and $Y$ are $d$ separated by $Z$. In particular, the Markov equivalence class of $G$ is generically identifiable from the spectrum of the process $\X$. 
\end{theorem}
Our second result generalises the main result in \cite{sullivant2010trek} to SVAR processes. An informal formulation of this result reads as follows.
\begin{theorem}[$t$-separation of processes]
    Let $G=(V,D)$ be a process graph and $X,Y \subset V$ be two sets of processes (indices). For a generic SVAR process compatible with $G$, the rank of the subspectrum $\sd_{X,Y}$ (as a matrix over the field of rational functions) is equal to the minimal set of processes that trek separate $X$ and $Y$ on $G$. 
\end{theorem}

\textsc{Causal effect identification}. Section \ref{sec: rational identifiability} is devoted to the identification of causal effects in the frequency domain. If a causal effect in the frequency domain can generically be expressed through addition/ subtraction and multiplication/division of the entries in the spectrum, then we call that effect rationally identifiable. In this paper, we make this notion more precise and, in so doing, generalise the rational identifiability term that is known for linear Gaussian SCMs. We will conclude this section by showing that the recent latent factor half-trek criterion for linear Gaussian SCMs \cite{10.1214/22-AOS2221} can be applied to the process graph to decide whether a given frequency domain causal effect is rationally identifiable. However, we need to require that the process graph is acyclic, while \cite{10.1214/22-AOS2221} allows also for cyclic graphs. 

\section{Preliminaries}\label{sec: preliminaries}
Let $V$ be a finite set of process indices and $\Z$ the set of integers. A \emph{stationary time series graph} is an infinite graph $\mathcal{G} = (V \times \Z, \mathcal{D})$ such that the set of directed edges $\D$ satisfies the following requirement: If there is a directed edge $(v(t-k), w(t)) \in \D$, then it must hold that $k \geq 0$ and $\{(v(s-k), w(s))\mid s \in \Z \} \subset \D$. In this case we say that $v$ is causing $w$ at lag $k$, which we abbreviate with the notation $v\to_k w$. If the lag $k=0$, then $v\to_0 w$ represents a \emph{direct contemporaneous effect}. The requirement $k \geq0$ makes sure that a cause cannot not happen after its effect.
The \emph{order} $p$ of the time series graph $\G$ is the maximum time lag of a link on $\G$, i.e. $p \coloneqq \max_{v,w\in V}\{k\mid v \to_k w \}$.

The time series graph encodes how the states of a dynamically evolving system of processes causally influence each other at and across discrete time points, i.e. through contemporaneous and lagged effects. Collapsing the infinite time series graph $\G$ along the time axis yields its finite \emph{process graph} $G=(V,D)$. This graph contains a directed edge $(v,w)\in D$, which we abbreviate as $v \to w$, if and only if there is $k \geq 0$ such that $v \to_k w$ is a contemporaneous or lagged causal relation on the time series graph $\G$. The parent set of some process $v\in V$ is denoted by the symbol $\pa(v)$ and consists of those processes $u$ that are connected to $v$ by a direct link, i.e. $u \to v$. For a directed graph $G=(V,D)$ we denote by $\T(G)$ the set of all time series graphs that have process graph $G$.

\begin{figure}
    \centering
    \includegraphics[width=0.7\textwidth]{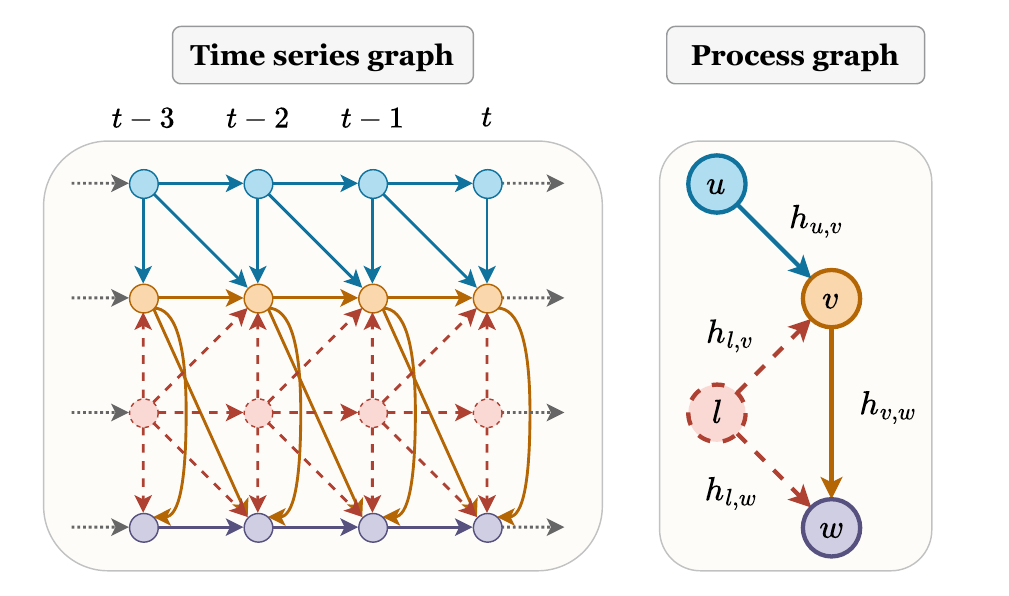}
    \caption{The figure on the left shows an excerpt of an order one time series graph for the processes $V=\{u,v,h,w \}$. The graph on the right displays the associated process graph.}
    \label{fig:tsg_example}
\end{figure}

A stochastic process $\X = (\X_v(t))_{v\in V, t \in \Z}$ is said to be a \emph{structural vector autoregressive (SVAR) process} compatible with a time series graph $\G$ if there exists a collection of coefficients $\Phi = (\phi_{v,w}(k))_{v \to_k w}$ indexed by the contemporaneous and lagged relations in $\D$, such that for every $v\in V$ and every $t \in \Z$ the associated random variable satisfies the structural equation 
\begin{align}
    \X_v(t) &= \sum_{u \in pa(v)} \sum_{k: u \to_k v} \phi_{u,v}(k) \X_u(t-k) + \eta_v(t), 
\end{align}
where $\eta = (\eta_v(t))_{v \in V, t \in \Z}$ is a mutually independent zero mean Gaussian white noise process, i.e. all $\eta_v(t)$ are independent zero mean normally distributed random variables. In what follows we denote by $\Omega=\diag((\omega_v)_{v\in V})$ the positive diagonal covariance matrix of the random vector $\eta(t)$. In particular, the variance does not depend on the time index $t$. Furthermore, we require the coefficients $\Phi$ to be \emph{stable}, see \cite{lutkepohl2005new}. 
\subsection{The frequency representation of a SVAR process}
In the previous section we interpreted an SVAR process as a parametrised model for an appropriate time series graph $\G$. In \cite{reiter2023formalising}, we show that an SVAR process can equivalently be written as a structural equation process on the process graph. This model contains a linear filter $\lambda_{v,w} = \left( \lambda_{v,w}(s) \right)_{s \in \Z }$ for each link $v \to w$ on the process a graph, and for every $v \in V$ a univariate process $\X^\internal_v = \left( \X_v^\internal(t) \right)_{t \in \Z}$ called the internal dynamics of $v$. These filters and processes are such that the internal dynamics are mutually independent, and for every $v \in V$ it holds that 
\begin{align}
    \X_v &= \sum_{u \in pa(v)} \lambda_{u,v} \ast \X_u + \X_v^{\internal}, 
\end{align}
where $\ast$ denotes convolution of sequences. If we arrange the filters $(\lambda_{u,v})_{u \to v}$ into a matrix valued filter $\Lambda$, then we obtain the following system of linear equations for the stochastic processes: 
\begin{equation} \label{eq: time SEP formulation}
    \begin{split}
        \X & = \Lambda^\top \ast \X  + \X^\internal \\
    & = (\sum_{k \geq 0} \Lambda^k)^\top \ast \X^\internal
    \end{split}
\end{equation}
This representation exists under certain conditions on $\Phi$, as detailed in \cite{reiter2023formalising}. 

The contemporaneous and lagged covariance information in the processes $\X$ resp. $\X^\internal$ are captured in their \emph{Auto Covariance Sequence (ACS)}, which we denote by $\Sigma = (\Sigma(k))_{k \in \Z}$ resp. $\Sigma^\internal = (\Sigma^\internal(k))_{k \in \Z}$. For $k \in \Z$ the corresponding entries in the ACS are as follows
\begin{align*}
    \Sigma(k) &\coloneqq \E[\X(t) \X(t-k)^\top] & \Sigma^\internal(k) &\coloneqq \E[\X^\internal(t) \X^\internal(t-k)].
\end{align*}
Since the internal dynamics are mutually independent it holds that $\Sigma^\internal_{v,w}(k) = 0 $ for every $k \in \Z$ if $v \neq w$. The structural process formulation (\ref{eq: time SEP formulation}) yields the following expression for the ACS \cite{reiter2023formalising}: 
\begin{align}\label{eq: acs parameterisation}
    \Sigma &= (\sum_{k\geq0}\Lambda^k)^\top \ast \Sigma^\internal \hat{\ast} (\sum_{k \geq 0} \Lambda^k)
\end{align}

Every stationary stochastic process admits an equivalent description in the frequency domain \cite{brockwell2009time}, and it turns out that the frequency domain is a convenient perspective to reason about the causal structure of SVAR process at the level of its process graph. The frequency domain analogue of the ACS is the \emph{spectral density} $\sd \coloneqq \F(\Sigma)$, where $\F$ is the Fourier transformation. Recall that the Fourier transform maps sequences to complex valued functions that are defined on the complex unit circle $S^1 \subset \C$. Elementary properties of the Fourier transformation together with Equation (\ref{eq: acs parameterisation}) give the identity
\begin{align} \label{eq: structured spectrum}
    \sd &= (I - \Fsep)^{-\top} \sdint (I - \Fsep^\ast)^{-1},
\end{align}
where $\Fsep = \F(\Lambda)$ is the \emph{direct transfer function}, as shown in \cite{reiter2023formalising}. The entries in the direct transfer function are rational functions parameterised by the SVAR coefficients $\Phi$. To write down this parameterisation, let us define for a pair of process indices $x,y \in V$ the associated \emph{lag  polynomial}
\begin{align*}
    \varphi_{x,y}(z) &\coloneqq \sum_{k : x \to_k y} \phi_{x,y}(k) z^k .
\end{align*}
Suppose $v \to w$ is a direct link on the process graph, then its associated \emph{link function} \cite{reiter2023formalising} evaluated at $z\in S^{1}$ is
\begin{align}\label{eq: transfer function}
    \fsep_{v,w}(z) &\coloneqq \F(\lambda_{v,w})(z) = \frac{\varphi_{v,w}(z)}{1 - \varphi_{w,w}(z)}. 
\end{align}
Also the entries in the spectral density of the internal dynamics $\sdint \coloneqq \F(\Sigma^\internal)$ evaluated at a $z \in S^1$ admit a direct parameterisation in terms of the SVAR coefficients. Specifically, the spectral density of the internal dynamics of $v \in V$ is parameterised by
\begin{equation}\label{eq: parameterisation internal spectrum}
    \sdint_v(z) = \frac{\omega_v}{(1- \varphi_{v,v}(z))(1- \varphi_{v,v}(z^\ast))},
\end{equation}
where $\omega_v$ is the variance of the white noise process $\eta_v$, see \cite{reiter2023formalising} .
The pairwise independence of the internal dynamics implies that $\sdint_{v,w}(z) = 0$ for every $z \in S^1$ if $v \neq w$.

\begin{example}
    Considering the example shown in Figure \ref{fig:tsg_example}, we illustrate the computation of the direct link function for the link $u \to v$ on the process graph. The coefficients $\phi_{u,v}(0)$ and $\phi_{u,v}(1)$ parameterise the contemporaneous and lag-one effect of $u$ on $v$. Similarly, the coefficient $\phi_{v,v}(1)$ quantifies the effect with which the process $v$ carries information from one point in time to the next. So, according to equation (\ref{eq: transfer function}), the transfer function associated with the edge $u\to v$ is computed as follows
    \begin{align*}
        \fsep_{u,v}(z) &= \frac{\phi_{u,v}(0) + \phi_{u,v}(1)z}{1 - \phi_{v,v}(1)z}
    \end{align*}
    Regarding the internal spectrum of $v$, we use equation (\ref{eq: parameterisation internal spectrum}) to compute  
    \begin{align*}
        \sdint_u(z) &= \frac{\omega_u}{1 - 2\phi_{u,u}(1)\mathrm{Re}(z) + \phi_{u,u}(1)^2},
    \end{align*}
    where $z = \mathrm{Re}(z) + i\mathrm{Im}(z) \in S^1$.  
\end{example}

\subsection{The spectral trek rule}
The algebraic relations between the entries in the spectrum are constrained by the process graph $G$. In the following, we recall how the process graph structures the spectrum using the \emph{spectral trek rule} \cite{reiter2023formalising}, a generalisation of the classical trek rule \cite{Wright1921CorrelationAndCausation, wright1934method}.
\begin{definition}
    A \emph{directed path} $\pi$ on $G$ is an ordered sequence of links $\pi = v_1 \to \cdots \to v_n$. The node $v_1$ is called the starting point of $\pi$, and the node $v_n$ is called the endpoint of $\pi$. A path $\pi'$ is said to be a \emph{sub-path} of $\pi$, denoted as $\pi' \leq \pi$, if there are $1 \leq k \leq m \leq n$ such that $\pi'= v_k \to \cdots \to v_m$. The empty path at $v\in V$ is the directed path given by the sequence $(v)$. If $v,w \in V$ are process indices, then $\mathbf{P}(v,w)$ refers to the set of directed paths that start at $v$ and end at $w$. 
    
    A \emph{trek} from $v$ to $w$ is a triple $\tau=(\Left(\tau), \ttop(\tau),  \Right(\tau))$, where $\ttop(\tau)\in V$ is a process index, and $\Left(\tau)\in \mathbf{P
    }(\ttop(\tau),v)$ and $\Right(\tau)\in \mathbf{P}(\ttop(\tau), w)$ are directed paths. The set of all treks from $v$ to $w$ is denoted as $\treks(v,w)$.  
\end{definition}
Suppose $\pi = (v_i)_{1 \leq i \leq n}$ is a directed path on $G$, then its associated \emph{path function} \cite{reiter2023formalising} is the point-wise product of the associated link functions 
\begin{align*}
    \fsep^{(\pi)} & \coloneqq \prod_{i=1}^{n-1}\fsep_{v_i, v_{i+1}}.
\end{align*}
Note that the path function of the empty path is the constant function.
The \emph{trek function} \cite{reiter2023formalising} of a trek $\tau$ is the point-wise product of functions
\begin{align*}
    \sd^{(\tau)} &\coloneqq \fsep^{(\Left(\tau)))} \sdint_{\ttop(\tau)} (\fsep^{(\Right(\tau)})^\ast.
\end{align*}
In the following, we recall how the entries in the spectral density can be computed in terms of trek functions. 
\begin{proposition}[Spectral trek rule \cite{reiter2023formalising}] \label{prop: trek rule}
    Suppose $v, w\in V$  are process indices on the process graph $G$, then the entry $\sd_{v,w}$ is equal to the sum of the trek functions that are associated with the treks from $v$ to $w$, i.e
    \begin{align*}
        \sd_{v,w} &= \sum_{\tau \in \treks(v,w)} \sd^{(\tau)}.
    \end{align*}
\end{proposition}
\begin{example}
    On the process graph in Figure \ref{fig:tsg_example} there is the directed path $\pi = u \to v \to w$. The path function of $\pi$ evaluated at some $z  \in S^1$ is 
    \begin{align*}
        \fsep^{(\pi)}(z)&= \fsep_{u,v}(z) \fsep_{v,w}(z) \\
        &= \frac{\phi_{u,v}(0)\phi_{v,w}(0) + (\phi_{u,v}(0)\phi_{v,w}(1) + \phi_{u,v}(1)\phi_{v,w}(0))z + \phi_{u,v}(1)\phi_{v,w}(1)z^2}{1-(\phi_{v,v}(1)+ \phi_{w,w}(1))z - \phi_{v,v}(1)\phi_{w,w}(1)z^2}.
    \end{align*}
    Furthermore, let us compute the entries $\sd_{u,w}$ and $\sd_{v,w}$ using the spectral trek rule. For $\sd_{u,w}$, note that there is exactly one trek connecting $v$ and $w$, namely the path $\pi$. Regarding $\sd_{v,w}$, the set of treks is $\treks(v,w) = \{ v \leftarrow l \to w, v \to w \}$. Consequently, we get the following expressions for the corresponding entries in the spectrum:
    \begin{align*}
        \sd_{u,w} &= \fsep^{(\pi)}\sdint_u & \sd_{v,w}&= \fsep_{v,w}\sdint_v + \fsep_{l,v}\sdint_l \fsep_{l,w}^\ast
    \end{align*}
\end{example}

\section{An algebraic parameterisation of SVAR processes in the frequency domain}\label{sec: frequnency domain parameterisation}
The aim of this section is to provide a precise algebraic formulation of how the SVAR parameters map onto the direct transfer function, onto the spectral density of the internal dynamics, and finally onto the spectral density of the process. We will construct all of these objects as matrices over the field of rational functions.  
\subsection{Rational functions}
In this section, we will first recall the space of rational functions, and then construct a map on the space of rational functions that we need to parameterise the frequency representation of SVAR processes over the field of rational functions with real coefficients. 

The field of rational functions with real coefficients is constructed from the ring of polynomials, i.e. 
\begin{align*}
    \R[z] \coloneqq \bigcup_{d\geq 0}\{ \sum_{k=0}^d \alpha_k z^k | \alpha_k \in \R, \alpha_d \neq 0 \}.
\end{align*} 
The \emph{degree} of a polynomial $f= \sum_{k=0}^d \alpha_k z^k$ is $\deg(f) \coloneqq \max\{k: \alpha_k \neq 0 \}$. The field of fractions of the ring $\R[z]$ are the rational functions with real coefficients
\begin{align*}
    \R(z) \coloneqq \{ \frac{f}{g} | f, g \in \R[z] \}_{/ \sim },
\end{align*}
where $\frac{f}{g} \sim \frac{f'}{g'}$ if $fg' = f'g$. The degree of a rational function $r = \frac{f}{g} \in \R(z)$, where $f$ and $g$ are coprime, i.e. they do not share a common divisor in $\R[z]$, is the maximum of $\deg(f)$ and $\deg(g)$. Since every rational function admits a unique representation as a fraction of polynomials that are coprime, the degree of a rational function is well defined. 

The calculation of the spectral density involves complex conjugation. This seems to prevent the parameterisation of the spectral density as an element over the field $\R(z)$. However, we can get around this obstacle because the relevant domain for the spectral density is only the complex unit circle $S^1$, and for any point $z \in S^1$ it holds that $z^\ast = z^{-1}$. Based on this observation, we will construct an involution on the rational functions, i.e. an algebra isomorphism $()^\ast: \R(z) \to \R(z)$, which is its own inverse, such that in addition $r^\ast(z) = r(z^\ast)$ for every $z \in S^1$. 

As a first step in the construction of this map, we provide an auxiliary map on the space of polynomials $()^\ast: \R[z] \to \R[z]$. This map sends a polynomial of degree $n$, i.e. $f = \sum_{k=0}^n\alpha_k z^k \in \R[z]$, to $f^\ast \coloneqq \sum_{k=0}^n \alpha_{n-k}z^k$. The map $()^\ast$ is not additive but it commutes with multiplication on $\R[z]$ as we show in Lemma \ref{lemma: multiplicativity conjugate polynomials} in the Appendix. This property will turn out to be essential for the subsequent construction of the map on the rational functions.
Based on the auxiliary map $()^\ast$ we define the conjugation $()^\ast: \R(z) \to \R(z)$ on the space of rational functions as follows: 
\begin{align}\label{eq: rational conjugation}
    \frac{f}{g} \mapsto \left( \frac{f}{g} \right)^\ast \coloneqq \left(z \mapsto  \frac{f^\ast(z)}{g^\ast(z)} z^{\deg(g) - \deg(f)}\right)  
\end{align}
\begin{example}
    As an example of conjugation on $\R(z)$ let us consider the link function for $u \to v$ on the process graph in Figure \ref{fig:tsg_example}. The conjugated rational link function is equal to
    \begin{align*}
        \fsep^\ast_{u,v}(z) &= \frac{\phi_{u,v}(1) + \phi_{u,v}(0)z}{-\phi_{v,v}(1) + z}
    \end{align*}
\end{example}
The conjugation map is indeed an automorphism on the space of rational functions $\R(z)$, i.e. it is linear and respects the multiplication of rational functions. The proof of the following Proposition is deferred to the Appendix.
\begin{proposition}\label{prop: conjugation of rational functions}
    The conjugation $()^\ast: \R(z) \to \R(z)$ is an algebra involution, i.e. conjugation respects multiplication and addition of rational functions, and the conjugation map is its own inverse map, i.e. $(r^\ast)^\ast = r$ for all $r \in \R(z)$. Finally, it holds for every $z \in S^1$ and every $r \in \R(z)$ that $r^\ast(z) = r(z^\ast)$.  
\end{proposition} 
Since $\R(z)$ is a field (with characteristic zero), we can consider vector spaces over $\R(z)$ with all known concepts from linear algebra, e.g. linear maps, linear independence, rank, bases, Gaussian elimination, determinants. An immediate implication of Proposition \ref{prop: conjugation of rational functions} is the following.
\begin{corollary}
    Let $A \in \R(z)^{n \times n}$ be a matrix over the field of rational functions with real coefficients, then it holds that 
    \begin{align*}
        \det(A^\ast) &= \det(A)^\ast \in \R(z),
    \end{align*}
    where $A^\ast$ is the entry-wise conjugation of $A$. In particular, $\det(A)= 0$ if and only if $\det(A^\ast) = 0$. 
\end{corollary} 
Throughout this work we use the following matrix notations. Suppose $M\in \mathbb{K}^{V \times V}$ is a matrix over some field $\mathbb{K}$ (of characteristic zero) and let $X,Y \subset V$ be subsets, then we write $[M]_{X,Y}$ for the submatrix of $M$ that one obtains by selecting the entries corresponding to the rows indexed by $X$ and the columns indexed by $Y$, i.e. $[M]_{X,Y} \coloneqq (M_{x,y})_{x\in X, y \in Y} \in \mathbb{K}^{X \times Y}$.  
\begin{remark}\label{remark: rank}
    Suppose $M \in \mathbb{K}^{X\times Y}$ is a matrix. A submatrix $M_r = [M]_{X_r, Y_r} \in k^{X_r \times Y_r}$ of $M$ is called an $r$-minor if the subsets $X_r\subset X$ and $Y_r \subset Y$ have cardinality $r$, i.e. $|X_r| = |Y_r| = r$. The matrix $M$ has rank $r$ if $M$ has an ivertibel $r$-minor $M_r$, i.e. $\det(M_r)) \neq 0$, and if every $r+1$ minor $M_{r+1}$ of $M$ is not invertible, i.e. $\det(M_{r+1}) = 0$.
\end{remark}

\subsection{The rational frequency SVAR representation}
In this subsection, we use the terminology and constructions from the previous subsection for the frequency domain parameterisation of SVAR processes. Our overall goal is to build an algebraic identifiability theory on top of this parameterisation. Since identifying causal relationships and effects becomes particularly interesting and challenging when some of the variables are hidden, we extend our graphical terminology to also include latent processes.

Suppose $\G$ is a time series graph over a set of process indices $V$, some of which are labelled observable and some of which are labelled latent, i.e. $V = O \cup L$ with $O \cap L = \emptyset$.
We define the set $\D_{O} \subset \D$ to be the set of all edges that emerge from some observed $v(t)\in O \times \Z$, and we denote as $\D_{L}$ the set of edges originating from a latent $l(t)\in L \times \Z$. Similarly, we divide the set of process-level links $D = D_O \cup D_L$, where $D_O$ consists of the links that emerge from an observed process and $D_L$ contains the links emerging from a latent process. For a process $v\in V$, we distinguish between its observed and latent parents by writing $\pa_O(v)$ and $\pa_L(v)$ for $\pa(v)\cap O$ and $\pa(v)\cap L$ respectively. Throughout this work require any latent structure to satisfy the following.
\begin{comment}
    In terms of notation, we distinguish the link functions of links in $D_O$ and $D_L$. That is, if $(v,w) \in D_O$, then we denote its link function by $\fsep_{v,w}$. On the other hand, if $(v,w) \in D_L$, then we denote its link function by $\lfsep_{v,w}$.
\end{comment} 
\begin{assumption}[Exogenous latent structure]\label{assumption: latent structure}
    Suppose $\G=((O \cup L) \times \Z, \D_O \cup \D_L)$ is a time series graph carrying a latent structure. Then the set of edges is such that every latent process $l \in L$ has zero incoming edges on the process graph $G=(O \cup L, D_O \cup D_L)$. 
\end{assumption} 
\begin{example}
    Let us equip the process graph in Figure \ref{fig:tsg_example} with a latent structure. We define the observed processes to be $O= \{u,v,w\}$ and the latent processes are $L = \{ l \}$. Accordingly, the set of observed edges are $D_O =\{ u\to v, v \to w \}$ and the set of latent edges are $L = \{l\to v, l \to w \} $. This aligns with Assumption \ref{assumption: latent structure} because there no edges pointing to $l$.  
\end{example}

\paragraph{The space of possible transfer functions} 
Let $\R(z)_\stable^{D}$ be the set of matrices $A\in \R(z)^{V \times V}$ that satisfy the following three conditions:
\begin{enumerate}
    \item If $v,w \in V$ and $v \not \to w$, then $A_{v,w} =0$.
    \item The entries of $A$ do not have poles inside the complex unit disk. 
    \item The submatrix $[I-A]_{O,O}\in \R(z)^{O \times O}$ is invertible and its inverse is
    \begin{align*}
        [I-A]^{-1}_{O,O} &= \sum_{k \geq 0} [A]_{O,O}^k,
    \end{align*}
    and the entries of $([I-A]_{O,O})^{-1}$ do not have poles inside $D^1$.
\end{enumerate} 
In order to make sure that the transfer function $H$ associated with an SVAR parameter $\Phi$ is indeed an element in $\R^D_\stable(z)$, we need to impose suitable constraints on $\Phi$.

\paragraph{SVAR parameter space} In the following we denote with $\R^\D$  the vector space in which each dimension corresponds to a unique lagged or contemporaneous relation $v \to_k w$ on the time series graph. Since the order of $\G$ is finite the vector space $\R^\D$ is finite dimensional. 
We define  $\R^\D_\stable$ to be the set of all parameters $\Phi \in \R^\D$ such that for every $v\in V$ it holds that
\begin{align}\label{condition: auto links}
    \sum_{k \geq 0 } |\phi_{v,v}(k)| < 1,
\end{align}
and for the coefficients of the direct lagged effects between the observed processes it holds that
\begin{align}\label{condition: observed effects}
    \sum_{v,w \in O} \sum_{k: v\to_k w} |\phi_{v,w}(k)| < 1.
\end{align}
These two requirements render the space $\R^\D_\stable$ an open semi-algebraic subspace of $\R^\D$. If the SVAR parameter $\Phi$ lies in $\R^\D_\stable$, then $\Phi$ is stable and the associated transfer function $H(\Phi)$ is an element in $\R(z)^D_\stable$, see \cite{reiter2023formalising}.
The decomposition of the links $\D = \D_O \cup \D_L$ factors the set $\R^\D_\stable \cong \R^{\D_O}_\stable \times \R^{\D_L}_\stable$ so that a SVAR coefficient vector splits as $\Phi = (\Phi_O, \Phi_L)$. From now on we refer by $\R^\G_\stable$ to the product space $\R^\D_\stable \times \R^{V}_+$, where $\R^V_+$ is the space of all positive definite diagonal matrices.  
\begin{remark}
    If the process graph $G$ is acyclic, then it is sufficient to require the inequality (\ref{condition: auto links}) for every process to ensure that the associated transfer function $\Fsep(\Phi)$ is in $\R(z)^D_\stable$. Because of the Assumption \ref{assumption: latent structure}, we did not need to include the lagged effects arising from the latent processes in the inequality (\ref{condition: observed effects}).
\end{remark}

Also the internal spectrum can be considered as a (diagonal) matrix over the rational functions $\R(z)$, as specified in the following Lemma. 
\begin{lemma}\label{lemma: internal spectrum}
    There is a diagonal matrix $\Tilde{\sd}^\internal \in \R(z)^{O \times O}$ such that its restriction to the complex unit circle is precisely  the internal spectrum $\sdint$ from (\ref{eq: parameterisation internal spectrum}), i.e. $\Tilde{\sd}^\internal(z) = \sdint(z)$ for all $z\in S^1$. 
\end{lemma}
\begin{proof}
    Let $v \in V$ be a process index, then recall that for $z \in S^1$ we defined 
    \begin{align*}
        \sdint_v(z) &= \omega_{v} \frac{1}{(1-\varphi_{v,v}(z))(1- \varphi_{v,v}(z^\ast))},
    \end{align*}
    where $\varphi_{v,v} \in \R[z]$ is the lag polynomial of $v$. Using conjugation on rational functions $\R(z)$ (\ref{eq: rational conjugation}) we now define
    \begin{align*}
        \Tilde{\sd}^\internal_{v,w} &= \begin{cases}
            \omega_v \frac{1}{1- \varphi_{v,v}} \left( \frac{1}{1- \varphi_{v,v}}\right)^\ast & \text{ if } v = w \\
            0 & \text{ if } v \neq w
        \end{cases}
    \end{align*}
    so that $\Tilde{\sd}^\internal \in \R(z)^{O \times O}_{\diag}$, and by Proposition \ref{prop: conjugation of rational functions} it holds that $\sdint(z) = \Tilde{\sd}^\internal(z)$ for every $z \in S^1$. 
\end{proof}
From now on we use $\sdint$ to identify the diagonal matrix over $\R(z)$ as constructed in Lemma \ref{lemma: internal spectrum}. 

In the following, we refer by $\pd_O(\R(z), D_{L})$ to the set of all matrices $B+C^\top C^\ast \in \R(z)^{O \times O}$, where $B \in \R(z)^{O \times O}$ is a diagonal matrix such that $B(z) \in \C^{O \times O}$ is positive definite for every $z \in S^1$, and where $C \in \R(z)^{L \times D}_\stable$ such that $C_{l,v}=0$ if $l \not \to v$ and such that $C$ has no poles inside $D^1$.
Based on the above observations we conclude with the map
\begin{displaymath}\label{map: frequency parameterisation}
    \F_\G = ([\Fsep]_{O,O}, \sdlint): \R^\G_\stable \longrightarrow \R(z)^{D_O}_\stable \times \pd_{O}(\R(z), D_L),
\end{displaymath}
which sends a SVAR parameter $(\Phi, \Omega)$ to the associated direct transfer function (of the observed processes) and \emph{projected internal spectrum}
\begin{align*}
    \sdlint & \coloneqq \sdint_{O} +  [\Fsep]_{L,O}^\top \sdint_L [\Fsep]_{L,O}^\ast,
\end{align*}
where $\sdint_O$ resp. $\sdint_L$ denotes the submatrix $[\sdint]_{O,O}$ and $[\sdint]_{L,L}$ respectively.
\begin{example}
    Let us consider the projected internal spectrum of the process parameterised by the time series graph in Figure \ref{fig:tsg_example}. Starting with the diagonal entries, the projected internal spectrum of the process $u$ is equal to its internal spectrum, since $u$ has no latent parents. The processes $v$ and $w$, on the contrast, both have the process $l$ as a latent parent. So, the contribution of $l$ to $v$ resp. $w$ needs to be added to the internal spectra when computing the projected internal internal spectra, i.e. 
    \begin{align*}
        \sdlint_u &= \sdint_v \\ 
        &=\omega_u(1-\varphi_{u,u})^{-1}(1-\varphi_{u,u})^{-\ast} \\
        \sdlint_v &= \sdint_v + \fsep_{l,v}\fsep_{l,v}^\ast\sdint_l \\
        &= \omega_v(1-\varphi_{v,v})^{-1}(1-\varphi_{v,v})^{-\ast} + \omega_l\fsep_{l,v}\fsep_{l,v}^\ast (1-\varphi_{l,l})^{-1}(1-\varphi_{l,l})^{-\ast} \\
        \sdlint_w &= \sdint_w + \fsep_{l,w}\fsep_{l,w}^\ast \sdint_l \\
        &= \omega_w(1-\varphi_{w,w})^{-1}(1-\varphi_{w,w})^{-\ast} + \omega_l\fsep_{l,w}\fsep_{l,w}^\ast (1-\varphi_{l,l})^{-1}(1-\varphi_{l,l})^{-\ast} 
    \end{align*}
    For the off-diagonal entries, note that only the entry corresponding to the pair $(v,w)$ is non-zero. This is because $\pa_L(u)\cap \pa_L(v)=\emptyset$ and $\pa_L(u) \cap \pa_L(w) =\emptyset$, but $\pa_L(v) \cap \pa_L(w)= \{l\}$. Consequently, 
    \begin{align*}
        \sdlint_{u,v}&= \sdlint_{u,w} = 0 \\
        \sdlint_{v,w} &= \omega_h\fsep_{l,v}\fsep_{l,w}^\ast (1-\varphi_{l,l})^{-1}(1-\varphi_{l,l})^{-\ast} 
    \end{align*}
\end{example}
In the following, we verify that the projected internal spectrum indeed maps to an element in $\pd_O(\R(z), D_L)$.
\begin{lemma}
    For the projected internal spectrum $\sdlint$ it holds that: (1.) $\sdlint(z)$ is positive definite for every $z \in S^1$, (2.) for two distinct $v,w \in O$ it holds that $\sdlint_{v,w} =0$ if $\pa_L(v) \cap \pa_L(w) = \emptyset$.
\end{lemma}
\begin{proof}
    First, observe that for two observed processes $v,w \in O$ the associated entry in the projected internal spectrum is given by the formula 
    \begin{align*}
        \sdlint_{v,w} &= \begin{cases}
            \sdint_{v,v} + \sum_{l \in \pa_L(v)} \fsep_{l,v}\fsep_{l,v}^\ast \sdint_l & \text{ if } v = w \\
            \sum_{l \in \pa_L(v) \cap \pa_L(w)} \sdint_l\fsep_{l,v} \fsep_{l,w}^\ast & \text{ if } v \neq w
        \end{cases}
    \end{align*}
    From this it follows that $\sdlint$ is symmetric after entry wise conjugation and also that $\sdlint_{v,w}= 0$ whenever $\pa_L(v)\cap \pa_L(w) = \emptyset$. It remains to show that, on the unit circle, the projected internal spectrum evaluates to a positive definite matrix. So let $x \in \C^O$ be a non-zero vector and $z \in S^1$ a point on the complex unit circle. Using the formula above and Proposition \ref{prop: conjugation of rational functions} and that $\sdint_v(z) > 0$ for every $z \in S^1$ we get that 
\begin{align*}
    (x^\ast)^\top \sdlint(z) x = \sum_{v \in O } \omega_v|x_v|^2 \sdint_v(z) + \sum_{l \in \pa_L(v)} |x_v\fsep_{l,v}(z)|^2\sdint_l(z) > 0,
\end{align*} 
which shows that $\sdlint(z)$ is positive definite. 
\end{proof}
\begin{remark}
    Using the trek rule (Proposition \ref{prop: trek rule}) we see that the projected internal spectrum is given by the trek monomials that are associated with the trivial treks and the treks that involve edges in $D_L$ only. 
\end{remark}

An immediate question concerning the paramterising map of the transfer function of the observed processes $[H]_{O,O}$ is whether the SVAR parameters of $\Phi_O$ for the lagged and instantaneous effects between the observed nodes can be uniquely recovered from $[\Fsep(\Phi_O)]_{O,O}$ almost everywhere.
\begin{lemma}\label{lemma: injectivity parameterisation}
    The parameterisation of the direct transfer function $[\Fsep]_{O,O}$ is generically injective. That means, if $A \subset \R^{\D_O}_\stable$ denotes the subset of parameter vectors $\Phi_O$ for which there is another $\Phi_O' \in \R^{\D_O}_\stable \setminus \{\Phi_O\}$ such that $[H(\Phi_O)]_{O,O} = [H(\Phi_O')]_{O,O}$. Then the set $A$ has Lebesgue measure zero. 
\end{lemma}
\begin{proof}
    The recovery of $\Phi_O$ from $[\Fsep(\Phi_O)]_{O,O}$ could only fail if there is a link $v \to w$ between observed processes $v,w\in O$ such that the denominator and the numerator of the associated transfer function are not coprime, i.e. if the polynomial $\varphi_{v,w}$ and the polynomial $1 - \varphi_{w,w}$ share a common divisor in $\R[z]$. Recall that two polynomials $f,g \in \R[z]$ have a common divisor in $\R[z]$ if and only if the resultant $\res(f,g)$ evaluates to zero. Considering $\varphi_{v,w}$ and $\varphi_{w,w}$ as $\Phi_O$-parameterised polynomials exhibits $\res(\varphi_{v,w}, 1- \varphi_{w,w})$ as a multi-variate polynomial in $\Phi_O$. So the a parameter $\Phi_O$ can only then not be uniquely recovered from $H(\Phi_O)$ if $\res_{D_O}(\Phi_O) = 0$, where 
    \begin{align*}
        \res_{D_O} \coloneqq\prod_{v \to w \in D_O}\res(\varphi_{v,w}, 1- \varphi_{w,w}) \in \R[\Phi_O]. 
    \end{align*}
    The polynomial $\res_{D_O}$ is non-zero because each of its factors $\res(\varphi_{v,w}, 1- \varphi_{w,w})$, with $v \to w \in D_O$, is a non-zero polynomial. It follows that the zero set $ \{ \Phi_O \in \R^\D_\stable \mid \res_{D_O}(\Phi_O) = 0\}$ is a closed lower-dimensional variety (manifold) lying inside the parameter space $\R^{\D_O}$ and is therefore a zero set with respect to the Lebesgue measure. Since the set $A$ is contained in this zero set, the claim follows. 
\end{proof}

Finally, using identity (\ref{eq: structured spectrum}) and the parameterisation $\F_\G$ we conclude this section with the following parameterisation of the spectrum 
\begin{align}\label{map: spectrum}
    \sd_\G: \R^\G_\stable \overset{\F_\G}{\longrightarrow} \R(z)^{D_O}_\stable \times \pd_O(\R(z), D_L) \overset{\sd_G}{\longrightarrow} \pd_{O}(\R(z)), 
\end{align}
where $\sd_G(\Fsep, \sdlint) = (I - \Fsep^\top)^{-1} \sdlint (I - \Fsep^\ast)^{-1}$. The set $\pd_{O}(\R(z))\subset \R(z)^{O \times O}$ consists of all matrices $C^\top B C^\ast $, where $C \in \R(z)^{O \times O}$ and $B \in \pd_O(\R(z), D_L)$. If there is no need to emphasise the time series graph $\G$ we will rather write $\sd$ instead of $\sd_\G$ to denote the spectral parametrisation. 
\begin{remark}\label{rem: rational spectum}
    For the sections to come it will be useful to keep in mind that the parameterisation map $\sd_\G$ can also be seen as a rational function in the indeterminates $\Phi, \Omega$ and $z$, i.e as a matrix in $\R(\Phi, \Omega, z)^{O \times O}$, where $\R(\Phi, \Omega, z)$ is the field of rational functions in $\Phi, \Omega$ and $z$ with real coefficients. 
\end{remark}

\section{Spectral process graph identification}\label{sec: structure identifiability}
In this section, we use the parametrization of the spectrum as a matrix over the rational functions $\R(z)$ to relate the graphical structure of the process graph to algebraic statements about the spectrum. Specifically, we will see how graphical separation statements about $G$ correspond to rank conditions (which are conditions that can expressed in terms of determinants, see Remark \ref{remark: rank}) on certain submatrices of the spectrum $\sd$. These characterisations are analogous to the results for linear Gaussian SCMs. 
 
\begin{definition}[Generic rank of a submatrix in the spectrum]
    Let $G=(V,D)$ be a process graph and $\G \in \T(G)$ be a time series graph, and let $X,Y \subset V$ be subsets, then $[\sd_\G]_{X,Y}$ is said to be of rank $r$ generically if the set of SVAR parameters $A\subset \R^\G_\stable$ for which this is not the case has Lebesgue measure zero. 
    The process-level subspectrum $[\sd_G]_{X,Y}$ is said to be of rank $r$ generically if $[\sd_\G]_{X,Y}$ is of rank $r$ generically for each $\G \in \T(G)$.  
\end{definition}
\begin{remark}
    The subspectrum $[\sd]_{X,Y}$ generically has rank $r$ if there exist subsets $X_r \subset X$ and $Y_r \subset Y$ each of size $r$ such that $\det([\sd]_{X_r, Y_r}) \neq 0$ in $\R(\Phi, \Omega, z)$, and $\det([\sd]_{X_{r+1}, Y_{r+1}}) =0$ for all subsets $X_{r+1}\subset X$ and $Y_{r+1}\subset Y$, both of size $r+1$. 
\end{remark}
Suppose $X,Y$ are subsets of equal size such that $[\sd]_{X,Y}$ is generically invertible. So we can find $(\Phi, \Omega) \in \R^\G_\stable$ such that the parameterised subspectrum  $[\sd(\Phi, \Omega)]_{X,Y} \in \R(z)^{X \times Y}$ is invertible. This however, does not mean that $[\sd(\Phi, \Omega; z)]_{X,Y} \in \C^{X\times Y}$ is invertible for every $z \in S^1$. So, one could ask the following question. 
\begin{question}
    For $X,Y \subset V$ such that $\rank([\sd]_{X,Y}) =r$ generically we consider the set 
    \begin{align*}
        A_{X,Y} \coloneqq \{ (\Phi, \Omega) \in \R^\G_\stable \mid \rank([\sd_\G(\Phi, \Omega, z)]_{X,Y}) \neq r \text{ for all $z \in S^1$} \}.
    \end{align*}
    Does $A_{X,Y}$ have measure zero? 
\end{question}
Below we illustrate how the requirement that the non-zero determinant of a subspectrum vanishes at some point on the unit circle introduces polynomial constraints on the SVAR parameters.  
\begin{example}
    We consider the process graph $G=(V,D)$, where $V= [3]$ and $D = \{ 1 \to 2, 1 \to 3 \}$. The underlying time series graph $\G$ has order one and is fully connected under the restriction of $G$ and that the maximum time lag is one. For the computation of $\sd_{2,3}$ we need the following parameterised functions  
    \begin{align*}
        \fsep_{1,2}(z) &= \frac{\phi_{1,2}(0) + \phi_{1,2}(1)z}{1- \phi_{2,2}(1)z} & \fsep_{1,3}(z) &= \frac{\phi_{1,3}(0)+ \phi_{1,3}(1)z}{1- \phi_{3,3}(1)z} \\
        \intsep_{1}(z) &= \frac{1}{1 - \phi_{1,1}(z)}
    \end{align*}
    It is clear that $[\sd_\G]_{2,3}$ has generically rank one. In this example we will furthermore see that $\sd_{2,3}$ is generically non-zero on the entire unit circle. To see that we determine all parameter constellations $\Phi$ for which $\sd_{2,3}$ has a zero on the unit circle $S^1$. The parameterised expression for $\sd_{2,3}$ is 
    \begin{align*}
        \frac{(\phi_{1,2}(0) + \phi_{1,2}(1)z)z^2 (\phi_{1,3}(1)+ \phi_{1,3}(0)z)}{(1 - \varphi_{2,2}(z)) (1 - \varphi_{1,1}(1)z) (z - \varphi_{1,1}(1)) (z - \varphi_{3,3}(1))}
    \end{align*}
    This expression is zero if an only if the numerator is zero. The numerator is zero for some $z \in S^1$ if and only if the following polynomial evaluates to zero:
    \begin{align*}
        n_{2,3} = p_{2,3}^{(0)} + p_{2,3}^{(1)}z + p_{2,3}^{(2)}z^2,
    \end{align*}
    where each $p_{2,3}^{(i)}$ is a polynomial in the coefficients $\Phi$. They are as follows 
    \begin{align*}
        p_{2,3}^{(0)} &= \phi_{1,2}(0) \phi_{1,3}(1) \\
        p_{2,3}^{(1)} &= \phi_{1,2}(1)\phi_{1,3}(1) + \phi_{1,2}(0)\phi_{1,3}(0)\\
        p_{2,3}^{(2)} &= \phi_{1,2}(1)\phi_{1,3}(0)
    \end{align*}
    If $n_{2,3}$ vanishes at $a \in S^{1}$, then either $a=\pm 1$ or $a = x + iy$ with $y \neq 0$. First, consider the case where $n_{2,3}(1)=0$, this means 
    \begin{align*}
        n_{2,3}(z) &= (1-z)(\hat{p}_{2,3}^{(0)} + \hat{p}_{2,3}^{(1)}) 
    \end{align*}
    From this equation we get the following conditions
    \begin{align*}
        p_{2,3}^{(0)} &= \hat{p}_{2,3}^{0} \\
        p_{2,3}^{(1)} &= \hat{p}_{2,3}^{1} - \hat{p}_{2,3}^{0} \\
        p_{2,3}^{(2)} & = -\hat{p}_{2,3}^{1}
    \end{align*}
    So the numerator $n_{2,3}$ vanishes at $1$ if and only if the the coefficients satisfy the following relation
    \begin{align*}
        p_{2,3}^{(0)} + p_{2,3}^{(1)} + p_{2,3}^{(2)} = 0.
    \end{align*}
    Using a similar approach we see that the numerator vanishes at $a=-1$ if and only if the coefficients satisfy the condition
    \begin{align*}
        p_{2,3}^{(1)} - p_{2,3}^{(1)} - p_{2,3}^{(2)} = 0.
    \end{align*}
    Finally, $n_{2,3}$ vanishes at some $a= x+iy$ with non-zero imaginary part $y$, then it vanishes also at the complex conjugate $a^\ast$. This implies that $n_{2,3}$ vanishes at some $a \in S^1$ with non-zero imaginary part if and only if 
    \begin{align*}
        p_{2,3}^{(0)} &= 1 & p_{2,3}^{(1)} &= -2 x & p_{2,3}^{(2)} &= 1  
    \end{align*}
    This shows that for generic choices of $\Phi$, the entry in the spectral density $\sd_{2,3}$ is non-zero on the entire unit circle. 
\end{example}

\begin{definition}[Trek separation \cite{sullivant2010trek}]
    Let $G=(V,D)$ be a process graph and $X,Y,Z_X, Z_Y \subset V$ not necessarily disjoint subsets. Then $(Z_X, Z_Y)$ is said to \emph{t-separate} $X$ from $Y$ if for every trek $\tau $ from a process $x\in X$ to a process $y \in Y$ either $\Left(\tau)$ intersects $Z_X$ or $\Right(\tau)$ intersects $Z_Y$.
\end{definition}

\begin{theorem}[Spectral trek separation]\label{thm: trek separation}
    Let $G$ is an acyclic process graph, and $X,Y \subset V$ not necessarily disjoint subsets, and $(\Phi, \Omega)\in \R^\G_\stable$ an SVAR-parameter pair, then $[\sd(\Phi, \Omega)]_{X,Y}$ has rank less than or equal to $r$ if and only if there exist subsets $Z_X, Z_Y \subset V$ for which $|Z_X| + |Z_Y|\leq r$, and such that $(Z_X, Z_Y)$ t-separates $X$ from $Y$. Generically $[\sd_G]_{X,Y}$ is of rank  
    \begin{align*}
         \min \{ |Z_X| + |Z_Y| \mid (Z_X, Z_Y) \text{ t-separates $X$ from $Y$} \}.
    \end{align*} 
\end{theorem}
\begin{proof}
    The proof of this theorem follows by exactly the same arguments as in the case of SEM's \cite{sullivant2010trek}. The main result on which the proof is based is a spectral version of the Gessel-Viennot Lemma, which we will show in the next subsection.
\end{proof}  
\subsection{Spectral characterisation of process-level d-separation}
It is well known that for linear Gaussian SCMs conditional independence between variables are generically equivalent to $d$-separation of the corresponding nodes on the causal graph. We will here see that this observation generalises to SVAR processes. For the following, let $G=(V,D)$ be an acyclic process graph with underlying time series graph $\G\in\T(G)$. We begin by recalling the graphical notion of $d$-separation between nodes on directed graphs.
\begin{definition}[$d$-separation \cite{pearl2009causality, peters2017elements}]
    Let $G= (V,D)$ be an acyclic process graph and suppose $X,Y, Z \subset V$ are subsets of processes. The set $Z$ $d$-separates $X$ and $Y$ if every (not necessarily directed) path between a node $x\in X$ and $y\in Y$ is going through a node $k$ that is either: (1.) a non-collider node that belongs to $Z$, or (2.) a collider node that is not ancestral to any process in $Z$, i.e $k \in V \setminus \mathrm{anc}(C)$. 
    The node $k$ is called a collider if the path $\pi$ contains a link $v \to k$ and another link $w \to k $. Otherwise $k$ is a non-collider node on $\pi$. 
\end{definition}
We intend to characterise $d$-separation statements about the process graph by algebraic relations, namely conditional independence relations between processes. To define the conditional independence of SVAR processes, i.e. Gaussian processes in our case, we use the conditional spectrum as discussed in e.g. Chapter 8 in \cite{brillinger2001timeseries}. 
\begin{definition}[Spectral conditional independence]
    Let $\X$ be a SVAR process specified by the parameter configuration $(\Phi, \Omega)\in \R^\G_\stable$, and let $X,Y,Z \subset V$ be disjoint sets of process indices. Then the subprocesses $\X_X$ and $\X_Y$ are said to be \emph{conditionally independent} given $\X_Z$ if 
    \begin{align*}
        [\sd(\Phi, \Omega)]_{X,Y \mid Z} &\coloneqq[\sd(\Phi, \Omega)]_{X,Y} -  [\sd(\Phi,\Omega)]_{X,Z} [\sd(\Phi, \Omega)]_{Z,Z}^{-1} [\sd(\Phi, \Omega)]_{Z,Y}
    \end{align*}
    is the zero element in $\R(z)^{X \times Y}$. 
\end{definition}
It turns out that conditional independence statement between sub-processes are equivalently characterised by a rank condition on the spectrum.
\begin{proposition}
    The processes $\X_X$ and $\X_Y$ are conditionally independent given $\X_Z$ if and only if the rank of $[\sd(\Phi, \Omega)]_{X \cup Z, Y \cup Z}$ is equal to $|Z|$.
\end{proposition}
\begin{proof}
    The proof goes along the same lines as the proof of Proposition 2.7 of \cite{SULLIVANT2008algebraic}. The only difference being that one reasons over the field $\R(z)$ instead of $\R$. 
\end{proof}
Note, if the time series graph $\G$ has order zero, then this notion of conditional independence boils down to the definition of conditional independence for normally distributed random variables \cite{SULLIVANT2008algebraic}.
We are going to use the following graphical result to characterise conditional independence of SVAR processes via $d$-separation on the process graph. 
\begin{lemma}[\cite{di2009t}]\label{lemma: d-t-separation}
    Let $G= (V, D)$ be a DAG, and $X, Y, Z\subset V$ disjoint sets of vertices. Then it holds that $Z$ $d$-separates $X$ and $Y$ if and only if there is a partition $Z = Z_X \cup Z_Y$ such that $(Z_X, Z_Y)$ $t$-separates $X \cup Z$ and $Y \cup Z$. 
\end{lemma}  
\begin{theorem}\label{thm: generically Markov}
    Let $G= (V,D)$ be an acyclic process graph and $X,Y,Z\subset V$. Then it holds for generic parameter choices $(\Phi, \Omega)$ that the sub-processes $\X_X$ and $\X_Y$ are conditionally independent given $\X_Z$ if and only if $Z$ $d$-separates $X$ from $Y$.  
\end{theorem}
\begin{proof}
     Suppose that the processes $\X_X$ and $\X_Y$ are independent given $\X_Z$ for generic parameter configurations. Hence, the generic rank of $[\sd_\G]_{X\cup Z,Y\cup Z}$ is equal to $|Z|$. Theorem \ref{thm: trek separation} now implies that there is a tuple of sets $(Z_X, Z_Y)$ that $t$-separates $X\cup Z$ and $Y\cup Z$, and that satisfies $|Z_X| + |Z_Y| = |Z|$. The requirement $Z \subset Z_X \cup Z_Y$ implies that $Z = Z_X \cup Z_Y$. It now follows from Lemma \ref{lemma: d-t-separation} that $Z$ $d$-separates $X$ and $Y$. 
     On the other hand, assume $X$ and $Y$ are $d$-separated by $Z$, then by Lemma \ref{lemma: d-t-separation} there is $(Z_X, Z_Y)$ which $t$-separates $X\cup Z$ and $Y\cup Z$ such that $Z = Z_X \cup Z_Y$. Consequently, it holds generically that $\rank([\sd_G]_{X \cup Z, Y \cup Z}) \leq |Z|$. On the other hand it must hold that $Z \subset Z_X' \cup Z_Y'$ for any $(Z_X', Z_Y')$ that $t$-separates $Z \cup X$ and $Z \cup Y$ so that $\rank([\sd_\G]_{X \cup Z, Y \cup  Z}) \geq |Z|$. This finishes the proof.   
\end{proof}
For a time series graph $\G\in \T(G)$, a parameter pair $(\Phi, \Omega) \in \R^\G_\stable$ is said to define a stationary distribution that is faithful with respect to the process graph $G$ if for any three disjoint sets $X,Y, Z \subset V$ it holds that $X,Y$ are $d$-separated by $Z$ if and only if $\X_X$ and $\X_Y$ are conditionally independent given $\X_Z$. Otherwise $(\Phi, \Omega)$ is said to define an unfaithful distribution with respect to $G$. Theorem \ref{thm: generically Markov} implies that the set of parameter pairs defining an unfaithful stationary distribution is a zero set with respect to the Lebesgue measure. In particular, it is generically possible to derive the Markov equivalence  class \cite{pearl2009causality, spirtes2000causation, peters2017elements} of the process graph from the spectrum of the process.

\subsection{Spectral Gessel Viennot Lemma}
The proof of the trek separation theorem in the original paper \cite{sullivant2010trek} is based on the classical Gessel-Viennot lemma \cite{lindström1973matroids, GESSEL1985binomial}. In the following we will see that this lemma generalises to SVAR processes represented in the frequency domain, i.e. we establish a graphical characterisation of $\det([(I-H)^{-1}]_{X,Y}) \in \R(z)$, where $X,Y \subset V$ are equal-sized subsets. This description is essential for the proof of Theorem \ref{thm: trek separation} and also for proving graphical identifiability criteria for the link functions of SVAR processes, as we will see in the next section. For the rest of this subsection we assume that the process graph $G$ is acyclic. 

Suppose $X = \{x_1, \dots, x_r\}$ and $Y= \{y_1, \dots, y_r\} \subset V$ are equal-sized sets of process indices. A system of directed paths from $X$ to $Y$ is a collection of directed paths $\Pi= (\pi_1, \dots, \pi_r)$ such that for any two distinct paths $\pi_i, \pi_j \in \Pi$ it holds that they do not have the same starting or end point. In particular, the system $\Pi$ gives rise to a bijection from $X$ to $Y$, which maps $x \in X$ to the endpoint $y \in Y$ of the unique path $\pi \in \Pi$ with $x$ as its starting point. Given the ordering of $X$ and $Y$, $\Pi$ induces a permutation $\alpha$ on the set $\{1, \dots, r\}$ where $\alpha(i)=j$ if the unique path $\pi \in \Pi$ starting at $x_i$ ends at $y_j$. In the following, we write $\sgn(\Pi)$ for the sign of the associated permutation $\alpha$.

A system of directed paths $\Pi$ is said to be \emph{non-intersecting} if any two paths in $\Pi$ do not share a single vertex, i.e. for any distinct $\pi, \pi' \in \Pi$ it holds that $V(\pi) \cap V(\pi') = \emptyset$, where $V(\pi)$ is the set of vertices visited by $\pi$. In the following we denote as $\mathbf{P}(X,Y)$ the set of all systems of non-intersecting directed paths from $X$ to $Y$, and for $\Pi \in \paths(X,Y)$ we define 
\begin{align*}
    \Fsep^{(\Pi)} &\coloneqq \prod_{\pi \in \Pi} \fsep^{(\pi)}.
\end{align*}
The set $\mathbf{P}(X,Y)$ is finite if $G$ is acyclic but can be infinite if $G$ contains cycles. If $X= \{x\}$ and $Y =\{y\}$, then we abbreviate $\mathbf{P}(\{x\},\{y\})$ by $\mathbf{P}(x,y)$, which is simply the set of all directed paths starting at $x$ and ending at $y$. 
\begin{proposition}[Spectral Gessel-Viennot Lemma]\label{prop: spectral gessel viennot lemma}
     Let $G$ be a directed acyclic process graph and $\G \in \T(G)$, and $X, Y \subset V$ two subsets of equal size. Then it holds that 
     \begin{align*}
         \det([I - \Fsep]^{-1}_{X,Y}) &= \sum_{\Pi \in \mathbf{P}(X,Y)} \sgn(\Pi)\Fsep^{(\Pi)}.
     \end{align*}
     In particular, $\det([I- \Fsep]_{X,Y}^{-1})$ is zero (considered as an element in the field of rational functions $\R(\Phi, z)$) if and only if every system of directed paths from $X$ to $Y$ has an intersection.
\end{proposition}
The classical Gessel-Viennot Lemma is a consequence of the spectral Gessel-Viennot Lemma. In fact, this follows from the restriction to SVAR processes that have order zero. 
An important implication of the Gessel-Viennot Lemma is a graphical criterion for deciding whether a given submatrix of the covariance of a linear Gaussian model is invertible or not. Analogously, one can apply the same criterion to the process graph to check whether a certain submatrix of the spectrum is generically invertible (as a matrix over the field of rational functions). The following proposition gives the exact criterion. However, in order to state it, we first need to introduce the necessary notations.  

Suppose $T = ( \tau_1, \dots, \tau_n )$ is a system of treks. The right-hand side of $T$ is the system of directed paths $\Right(T) = (\Right(\tau_i))_{i \in [n]}$, and the left-hand side is the system $\Left(T)=(\Left(\tau_i))_{i \in [n]}$. Finally, the top of the system $T$ is given by the vertices $\ttop(T)= \{\ttop(v_i)\}_{i \in [n]}$. A system of treks is said to have a sided intersection if its right or left hand side is a system of intersecting directed paths. For subsets $X, Y\subset V$ we denote by $\mathbf{T}(X,Y)$ the set of all such trek systems $T$ that satisfy $\Right(T) \in \mathbf{P}(\ttop(T), Y)$ and $\Left(T) \in \mathbf{P}(\ttop(\T), X)$. That means, the right and left sides are systems of non-intersecting directed paths, i.e. $\treks(X,Y)$ is the set of trek systems without sided intersection.
  
\begin{proposition}[Spectral version of Proposition 3.4 in \cite{sullivant2010trek}]\label{prop: invertible subspectrum}
    Let $X, Y$ be subsets of $V$ with equal cardinality $r$. It then holds that $\det([\sd]_{X,Y}) = 0 \in \R(\Phi, \Omega, z)$ if and only if $\mathbf{T}_0(X,Y) = \emptyset$, i.e. there is no trek from $X$ to $Y$ without sided intersection. 
\end{proposition}
\begin{proof}
    For the proof of this Proposition it will be useful to expand the determinant $\det([\sd]_{X,Y})$ using the Cauchy-Binet formula, i.e. 
    \begin{align}\label{eq: cauchy-binet expansion spectrum}
        \det([\sd]_{X,Y}) &= \sum_{S \subset V : |S|=r} \det([(I- H)^{-1}]_{X,S}) \det([\sdint]_S) \det([(I - H^\ast)^{-1}]_{S,Y}).
    \end{align}
    This expansion implies that $\det([\sd]_{X,Y})$ considered as a rational function in $\R(\Phi, \Omega, z)$ is zero if and only if every summand is zero, and each summand is zero if and only if either $\det([(I-H)^{-1}]_{X, S})$ or $\det([(I-\Fsep)^{-1}]_{S,Y})$ is zero.
    
    First, we assume that $\det([\sd]_{X,Y}) = 0$, and we wish to show that any system of treks $T$ from $X$ to $Y$ has a sided intersection.  
    If $|\ttop(T)| < |X|$, then it follows that both $\Left(T)$ and $\Right(T)$ must be systems of intersecting directed paths and, therefore, $T$ has a sided intersection. If, on the other hand, $|\ttop(T)| = |X|$, then it follows by the above expansion that $\det([(I-H)^{-1}]_{X, \ttop(T)})$ or $\det([(I-H^\ast)^{-1}]_{\ttop(T), Y})$ must be zero. So Proposition \ref{prop: spectral gessel viennot lemma} implies that $\Left(T)$ or $\Right(T)$ must be a system of intersecting directed paths, and we conclude $\mathbf{T}_0(X,Y) = \emptyset$. 

    Now suppose that every system of paths from $X$ to $Y$ has a sided intersection. We show by contradiction that $\det([\sd]_{X,Y})$ is zero. Suppose there is a set $S$ such that both $\det([(I-H)^{-1}]_{X,S})$ and $\det([(I-\Fsep)^{-1}]_{S,Y})$ are non-zero, then according to Proposition \ref{prop: spectral gessel viennot lemma} there must be at least one trek $T \in \mathbf{T}_0(X,Y)$ with $\ttop(T) = S$. Since this contradicts the assumption that every system of treks from $X$ to $Y$ has a sided intersection, we conclude that $\det([\sd]_{X,Y})$ is the zero element in $\R(\Phi, \Omega, z)$.
\end{proof}

For subsets $X, Y \subset V$ of equal size $r$, the combination of equation (\ref{eq: cauchy-binet expansion spectrum}) and Proposition \ref{prop: spectral gessel viennot lemma} yields  
\begin{align}\label{eq: spectral sum of trek functions}
    \det([\sd]_{X,Y}) &= \sum_{T \in \treks(X,Y)} \sgn(T) \sd^{(T)}, 
\end{align}
where for $T= (\tau_1, \dots, \tau_r) \in \treks(X,Y)$ a system of treks without sided intersection and $\sd^{(T)} = \prod_{i=1}^r\sd^{(\tau_i)}$ the associated product of trek functions and $\sgn(T)$ the sign of the permutation associated with $T$.  

\section{Rational identifiability of link functions}\label{sec: rational identifiability}
In this section, we introduce a notion with which to reason about whether the direct transfer function, i.e. causal effects in the frequency domain, can be identified as rational expressions of the entries in the spectrum. With this notion it is then possible to generalise graphical identifiability criteria for linear Gaussian causal models to the frequency domain representation of SVAR processes at the level of the process graph. As an illustration, we show that the latent factor half-trek criterion \cite{10.1214/22-AOS2221}, one of the most recent graphical identifiability criteria, can also be applied to the process graph to decide whether certain link functions are generically identifiable. 

In this section, we consider time series graphs equipped with a latent structure that satisfy Assumption \ref{assumption: latent structure}. In particular, the set of process indices $V$ is partitioned into a set of observed processes $O$ and a set of latent processes $L$, i.e. $V = O \cup L$ and $L \cap O = \emptyset$. In principal, we would like to call a link between two observed process identifiable if there is rational function that maps the spectrum of the process to the link function of that edge, and the form of this rational expression should be independent of the underlying time series graph. Such a function can always be found if there are no latent processes. In this case the link functions are identified through regression on the parent processes, i.e. for every $w\in V$ it holds that
\begin{align*}
   [\Fsep]_{\pa(w),w} &= [\sd]_{\pa(w), \pa(w)}^{-1} [\sd]_{\pa(w),w}.
\end{align*}
In the scenario where latent processes confound the relations between observed processes it depends on the shape of the process graph whether the link functions can be recovered from the spectrum for almost all parameter configurations (that define a stable SVAR process). 
\begin{example}\label{ex: instrumental process}
    In the process graph of Figure \ref{fig:tsg_example}, it holds that the link function for the link $\fsep_{v,w}$ can be identified for almost all parameter choices through a rational expression on the spectrum. This follows from the identity
    \begin{align*}
        \fsep_{v,w} &= \frac{\sd_{w,u}}{\sd_{v,u}} = \frac{\fsep_{v,w}\fsep_{u,v}\sdlint_u}{\fsep_{u,v}\sdlint_u}, 
    \end{align*}
    which holds for every SVAR parameter configuration that yields a non-zero function $\sd_{u,v}$. Due to Theorem \ref{thm: trek separation} the term $\sd_{u,v}$ is generically non-zero, i.e. $h_{v,w}$ is identifiable for almost all SVAR parameter configurations (that define a stable process). In particular, we found a rational expression for the link function $\fsep_{v,w}$ in the entries of the spectrum $\sd$ that holds irrespective of the underlying time seeries graph. 
\end{example}
The following definition is based on the definition of \emph{rational identifiability} as used in e.g. \cite{10.1214/22-AOS2221, foygel2012half, drton2018algebraic}.
\begin{definition}[Rational identifiability]
    A link $v \to w$ between two observed processes $v,w \in O$ is called \emph{rationally identifiable} if there is a rational function $\psi_{v,w}: \pd_{O, \stable}(\R(z)) \to \R(z)$, and for every time series graph $\G \in \T(G)$ a proper algebraic set $A_\G \subset \overline{\R^\G_\stable}$ such that $\psi \circ \sd_\G(\Phi, \Omega) = \fsep_{v,w}(\Phi)$
    for all $(\Phi, \Omega) \in \R^\G_\stable \setminus A_\G$. In particular, if the link $v \to w$ is rationally identifiable, then it is identifiable for almost all SVAR parameter configurations.  
\end{definition}
 A link is rationally identifiable if its link function can be expressed through addition, multiplication and division of the entries (rational functions) in the spectrum, as we saw in Example \ref{ex: instrumental process}. The requirement that $\psi_{v,w}$ is a rational function defined on $\pd_{O, \stable}(\R(z))$ makes sure that the expression with which to identify the link function $\fsep_{v,w}$ is independent of the underlying time series graph. In particular, we do not need to know the specific time series graph in order to identify the link function. But if we happen to know the time series graph $\G \in \T(G)$ as well, then Lemma \ref{lemma: injectivity parameterisation} implies that in the generic case the parameters $\phi_{v,w} =(\phi_{v,w}(k))_{k: v \to_k w}$ and $\phi_{w,w}=(\phi_{w,w}(j))_{j:w \to_j w}$ can also be uniquely recovered.

\begin{comment}
    The process graph $G$ is said to be \emph{rationally identifiable} if there is a map $\psi : \pd_{O}(\R(z)) \to \R^{D_O}(z) \times \pd_O(\R(z), D_L)$, and for every time series graph $\G\in \T(G)$ a proper algebraic set $A_\G \subset \overline{\R^\G_\stable}$ such that
    $\psi(\sd_\G(\Phi, \Omega)) = \F_\G(\Phi, \Omega)$ for all $(\Phi, \Omega) \in \R^\G_\stable \setminus A_\G$. 
\end{comment}

Before we proceed with the formal latent factor half-trek criterion we revisit an example from \cite{10.1214/22-AOS2221} to motivate the idea behind it.
\begin{figure}
    \centering
    \includegraphics[width=0.65\textwidth]{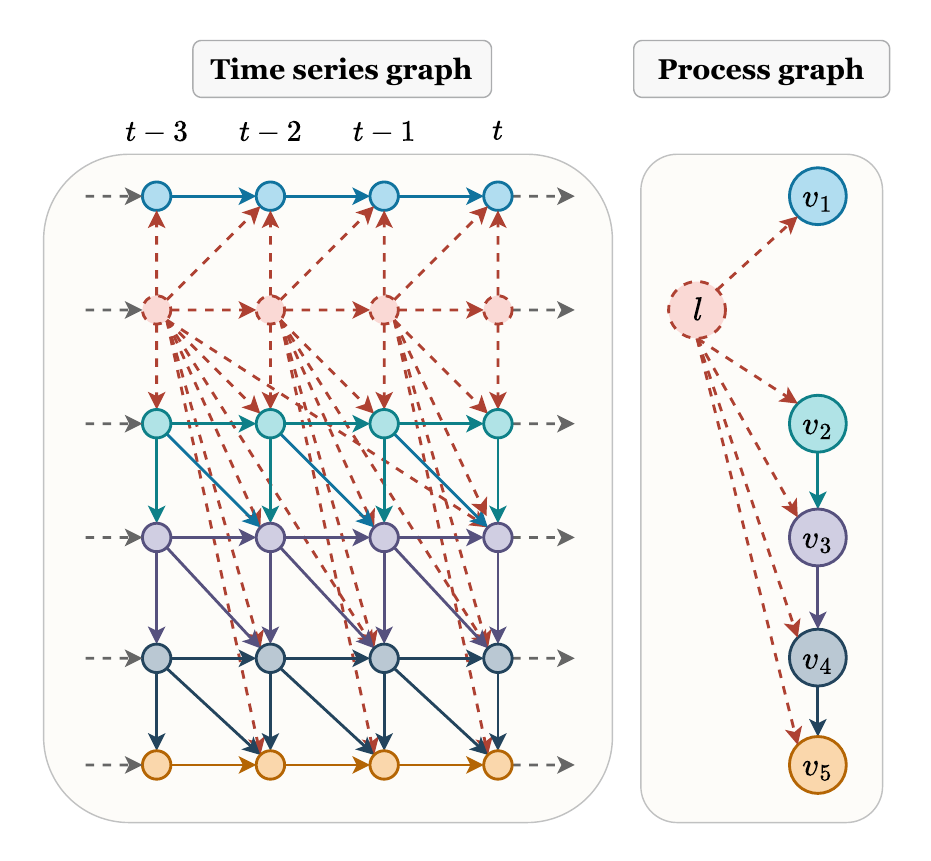}
    \caption{A time series graph (left) together with its process graph (right), see Figure 3.1 a.) in \cite{10.1214/22-AOS2221}. All link functions between the observed processes (represented by the blue nodes) are rationally identifiable. Furthermore, the graph satisfies the latent factor half-trek criterion}
    \label{fig: lfhtc_example}
\end{figure}

\begin{example}[Time series version of Example 3.1 in \cite{10.1214/22-AOS2221}]\label{ex: rational identifiability}
    We consider an SVAR process $\X$ consisting of six processes. The causal structure between these processes is specified by the process graph and its underlying time series graph, as shown in Figure \ref{fig: lfhtc_example}. Furthermore, the process has a latent structure, i.e. there are five observed processes indexed by $O = \{v_1, \dots, v_5\}$ and one latent process $h$. Since the latent process affects each observed process, none of the links between observed processes can be computed by ordinary regression because
    \begin{align*}
        \sd_{v_{i+1}, v_{i}} &= \fsep_{v_{i}, v_{i+1}} \sd_{v_i, v_i} +  \sdlint_{v_{i+1}, v_i}. 
    \end{align*}
    Nonetheless, every link between observed processes can be identified through rational operations on the spectrum. We begin with the link function $\fsep_{v_3, v_4}$. The basic idea is now to use the fact that all the observed processes are confounded by a common latent process. This induces relations that can be exploited. To see this, we employ the trek rule and compute  
    \begin{align*}
        \sd_{v_4, v_2} &= \fsep_{v_4, v_3}\sd_{v_3,v_2} + \sdlint_{v_4,v_2} \\
        \sd_{v_4, v_3} &= \fsep_{v_4, v_3}\sd_{v_3, v_3} + \sdlint_{v_4, v_3} + \sdlint_{v_4, v_2} \fsep_{v_2, v_3}^\ast
    \end{align*}
    In addition, we use the trek rule to observe the relations
    \begin{align*}
        \frac{\fsep_{l,v_4}}{\fsep_{l,v_1}} \sd_{v_1, v_2} &= \sdlint_{v_4, v_2} & \frac{\fsep_{l,v_4}}{\fsep_{l,v_1}} \sd_{1,3} &= \sdlint_{v_4, v_3} + \sdlint_{v_4, v_2}\fsep_{v_2, v_3}^\ast.
    \end{align*}
    Finally, these relations can be summarised by the linear equation system over the rational functions  
    \begin{align*}
        \begin{bmatrix}
            \sd_{v_4, v_2} \\ \sd_{v_4, v_3}
        \end{bmatrix} &= \begin{bmatrix}
            \sd_{v_3,v_2} & \sd_{v_1, v_2} \\ \sd_{v_3,v_3} & \sd_{v_1,v_3} 
        \end{bmatrix} \begin{bmatrix}
            \fsep_{v_3, v_4} \\ \alpha_{v_3,v_4}
        \end{bmatrix} & \alpha_{v_3,v_4} &= \frac{\fsep_{l,v_4}}{\fsep_{l,v_1}}.
    \end{align*}
    The matrix on the right-hand side of the equation is generically invertible as can be seen from Proposition \ref{prop: invertible subspectrum} and the fact that $\{v_2 \leftarrow h \to v_1 ,  v_3 \} $ is a system of treks from $\{v_2, v_3\}$ to $\{ v_3, v_1 \}$. This means that for generic parameter choices it is possible to obtain the vector on the right by matrix inversion, i.e. rational operations on the spectrum. In particular, the link function $\fsep_{v_3,v_4}$ is rationally identifiable. 
    Analogously, it is possible to see that the link function $\fsep_{v_4, v_5}$ is rationally identifiable. 

    We now consider the link function $\fsep_{v_2, v_3}$. Once more we use the trek rule to identify the equations
    \begin{align*}
        \sd_{v_3, v_1} &= \fsep_{v_2, v_3} \sd_{v_2,v_1} + \sdlint_{v_3, v_1} &
        \sd_{v_3, v_2} &= \fsep_{v_2, v_3} \sd_{v_2, v_2} + \sdlint_{v_3, v_2}.
     \end{align*}
    In order to apply a similar argument as before we would like to express the second terms in the sum as a linear factor of something that is rationally identifiable. To do this we observe that 
    \begin{align*}
        \sd_{v_4, v_1} - \fsep_{v_3, v_4} \sd_{v_3, v_1} &= \sdlint_{v_4, v_1} &
        \sd_{v_4, v_2} - \fsep_{v_3, v_4} \sd_{v_3, v_2}  &= \sdlint_{v_4, v_2}.
    \end{align*}
    From what we have seen above, the expressions on the left hand side, i.e. $\sdlint_{v_4, v_1}$ and $\sdlint_{v_4, v_2}$ respectively, are rationally identifiable, and these two terms are linearly related to $\sdlint_{v_3, v_1}$ and $\sdlint_{v_3, v_2}$ respectively. So we end up with the system of linear equations 
    \begin{align*}
        \begin{bmatrix}
            \sd_{v_3, v_1} \\ \sd_{v_3, v_2}
        \end{bmatrix} &= \begin{bmatrix}
            \sd_{v_2,v_1} & \sdlint_{v_4, v_1} \\ \sd_{v_2, v_2}  & \sdlint_{v_4, v_2}
        \end{bmatrix} \begin{bmatrix}
                \fsep_{v_2, v_3} \\ \alpha_{v_2, v_3}
            \end{bmatrix} & \alpha_{v_2, v_3} &= \frac{\fsep_{l,v_3}}{\fsep_{l,v_4}}
    \end{align*}
    Using similar arguments as above, we conclude that the matrix is generically invertible, so that the link function $\fsep_{v_2, v_3}$ is rationally identifiable. This shows that all links between observed processes are rationally identifiable. 
\end{example}

The idea behind the above procedure, i.e. exploiting linear relations between the entries in the projected internal spectrum, can be generalised by using both graphical and algebraic arguments, leading us to the graphical latent-factor half-trek criterion mentioned earlier. In order to formulate it and generalise it to the time series setting, we need to recall the necessary graphical notions. A \emph{latent factor half-trek} \cite{10.1214/22-AOS2221} on $G$ is a trek between two observed processes $v,w \in O$ that has one of the following shapes 
\begin{align*}
    v \leftarrow l \to x_1 \to \cdots \to x_n \to w & & v \to x_1 \to  \cdots \to  x_n \to w
\end{align*}
with $x_1, \dots, x_n \in O$ and $l \in L$. 

\begin{definition}[Latent half-trek reachability \cite{10.1214/22-AOS2221}]
    Let $v,w\in O$ be two observed processes and $L' \subset L$ a set of latent processes. If there exists a latent factor half-trek from $v$ to $w$ that does not pass through any node in $L'$, then $w$ is called \emph{half-trek reachable} from $v$ while avoiding $L'$. The set of processes that are half-trek reachable from $v$ while avoiding $L'$ is denoted $\htr_{L'}(v)$. If $X\subset O$ is a set of observed processes, then $\htr_{L'}(X) = \bigcup_{x \in X}\htr_{L'}(x)$ is the set of processes that are half-trek reachable while avoiding $L'$ from at least one process $x \in X$. 
\end{definition}

\begin{definition}[LF-HTC \cite{10.1214/22-AOS2221}]\label{def: LF-HTC}
    Let $v\in O$ be an observed process, and $Y,W \subset O \setminus \{v\}$ two sets of observed processes and $ L' \subset L$ a set of latent processes. The triple $(Y,W, L')$ is said to satisfy the \emph{latent factor half trek criterion (LF-HTC)} with respect to $v$ if it satisfies the following three conditions  
    \begin{enumerate}
        \item $|Y| = |\pa_O(v)| + |L'|$ and $|W| = |L'|$  with $W \cap \pa_O(v) = \emptyset$. 
        \item $Y \cap W =  \emptyset$ and $\pa_L(Y) \cap \pa_L(W \cup \{v \}) \subset L'$. 
        \item There exists a system of latent-factor half-treks from $Y$ to $\pa_O(v) \cup W$ with no sided intersection, such that any half-trek pointing to $w \in W$ has the form $y \leftarrow l \to w$ for some $y \in Y$ and $l \in L'$.  
    \end{enumerate}
\end{definition}
Let $T \in \treks(X,Y)$ be a system of treks from $X$ to $Y$. Then let $\G(T) = (V \times \Z, \D(T))$ be the time series subgraph of $\G$, where $v\to_k w$ is a causal relation on $\G(T)$ if and only if $v \to_k w$ on $\G$ and $v \to w$ is a link that appears in the system of treks $T$. Accordingly, $G(T)$ denotes the process graph of the time series graph $\G(T)$.
Since $\D(T)$ is a subset of $\D$, there is a canonical embedding $\iota_{T}:\R^{\D(T)}_\stable \to \R^\D_\stable$.
If $T\in \treks(X,Y)$ is such that $G(T)$ acyclic, then for $\Phi \in \R^{\D(T)}_\stable$ one can apply equation (\ref{eq: spectral sum of trek functions}) to get that 
\begin{align*}
        \det([\sd(\iota_T(\Phi), \Omega,z)]_{X,Y}) &= \sum_{T'\in \treks_{G(T)}(X,Y)} \sgn(T)\sd^{(T)}(\iota_T(\Phi), \Omega),
\end{align*}
where $\treks_{G(T)}(X,Y)$ is the finite non-empty set of non-intersecting trek systems from $X$ to $Y$ on the subgraph $G(T)$.
So it follows from Proposition \ref{prop: invertible subspectrum} that $\det([\sd]_{X,Y})\in \R(\Phi, \Omega, z)$ is non-zero. 

The following generalises the main Theorem in \cite{10.1214/22-AOS2221} to SVAR processes represented in the frequency domain. 
\begin{theorem}[spectral LF-HTC identifiability]\label{thm: LF-HTC}
    Suppose the triple $(Y, W, L')$ of subsets satisfies the LF-HTC with respect to an observed process $v \in O$. If all directed link functions associated with the links $u \to y \in D_O$ such that $y \in W \cup (Y \cap\htr_{L'}(W \cup \{v\} ))$ are rationally identifiable, then all direct link functions $\fsep_{x, v}$, with $x\in \pa_O(v)$, are rationally identifiable. 
\end{theorem}
\begin{example}
    Theorem \ref{thm: LF-HTC} allows us to prove, by purely graphical arguments, that every link in the process graph shown in Figure \ref{fig: lfhtc_example} is rationally identifiable, as we have already seen in Example \ref{ex: rational identifiability}. First, for $v_1$ and $v_2$ the triple $(\emptyset, \emptyset, \emptyset)$ satisfies the LF-HTC because $\pa_O(v_1)= \pa_O(v_2)=\emptyset$. Then we continue with the process $v_4$, for which we set $Y_{v_4} =\{v_2, v_3\}$ and $W_{v_4} = \{v_1\}$ and $L'_{v_4}= \{l\}$. These sets clearly satisfy conditions 1 and 2 of the definition \ref{def: LF-HTC}. To see that the third condition is also satisfied, note that $\{v_2 \leftarrow l \to v_1, v_3, \}$ is a system of treks from $\{v_2, v_3\}$ to $\{v_1, v_3\}$ with no sided intersection. Finally, since $Y_{v_4} \cap \htr_{L'_{v_4}}(\{v_1, v_4)\} = \emptyset$, it follows that the link $v_3 \to v_4$ is rationally identifiable. For the process $v_5$ one can choose exactly the same triple of sets to infer that $v_4 \to v_5$ is rationally identifiable. Finally, for the process $v_3$ we choose $Y_{v_3} = \{v_2, v_1\} $ and $W_{v_3} = \{ v_4 \}$ and $L'_{v_3} = \{l\}$. Again, conditions 1 and 2 of Definition \ref{def: LF-HTC} are easy to verify. The system of latent-factor half-treks required by condition 3 is $\{v_1 \leftarrow l \to v_4, v_2 \to v_3\}$. Since the link $v_3 \to v_4$ is rationally identifiable, it follows by Theorem \ref{thm: LF-HTC} that $v_2 \to v_3$ is rationally identifiable. 
\end{example}
\begin{proof}[Proof of Theorem \ref{thm: LF-HTC}]
    Let $\G \in \T(G)$ be some time series graph and let $(\Phi, \Omega)\in \R^\G_\stable$. We seek to construct the map $\Psi_{\pa_O(v), v}: \pd_{O}(\R(z)) \to \R(z)^{\pa_O(v)}$ which maps the spectrum $\sd_\G(\Phi, \Omega)$ to the vector of link functions $\Fsep(\Phi, \Omega)_{\pa_O(v), v} \coloneqq [\fsep_{u,v}]_{u \in \pa_O(v)} \in \R(z)^{\pa_O(v) \times \{v\}}$. In the following we will write $\sd$ to refer to $\sd_\G(\Phi, \Omega)$ and $\Fsep$ to refer to $\Fsep(\Phi)$. 
    
    The  construction of this map will go along the same lines as in the proof of the LF-HTC identifiability result in \cite{10.1214/22-AOS2221}. The first step is to establish over the field $\R(z)$ the linear relationship
    \begin{align}\label{eq: LF-HTC equation}
        \begin{bmatrix}
            \mathcal{A} & \mathcal{B}
        \end{bmatrix} \cdot \begin{bmatrix}
            \Fsep_{\pa_O(v), v} \\
            f 
         \end{bmatrix} &= g,
    \end{align}
    where $\mathcal{A} \in \R( z)^{Y \times \pa_O(v)}$, and $\mathcal{B} \in \R(z)^{Y\times W}$, and $f \in \R(z)^W$, and $g \in \R(z)^Y$. We begin with the definition of the matrices $\mathcal{A}$ and $\mathcal{B}$. For $y \in Y$ and $u \in \pa_O(v)$, we define the entry  
    \begin{align*}
        \mathcal{A}_{y, u} &\coloneqq \begin{cases}
            [(I  - \Fsep)^\top \sd]_{y, u} & \text{if } y \in \htr_{L'}(W \cup \{v\}) \\
            \sd_{y,u} & \text{if } y \not\in \htr_{L'}(W \cup \{v\})
        \end{cases}, 
    \end{align*}
     and for $y \in Y$ and $w \in W$ the corresponding entry in the matrix $\mathcal{B}$ is
    \begin{align*}
        \mathcal{B}_{y, w} &\coloneqq \begin{cases}
            [(I - \Fsep)^\top \sd (I - \Fsep^\ast)]_{y, w} & \text{if } y \in \htr_{L'}(W \cup \{v\}) \\
            [\sd (I - \Fsep^\ast)]_{y, w} & \text{if } y \not\in \htr_{L'}(W \cup \{v\})
        \end{cases}
    \end{align*}
    Finally, the component of the vector of rational functions indexed by $y \in Y$ is given as 
    \begin{align*}
        g_y = \begin{cases}
            [(I- \Fsep)^\top \sd]_{y,v} & \text{if } y \in \htr_{L'}(W \cup \{v\}) \\
            \sd_{y,v} & \text{if } y \not\in \htr_{L'}(W \cup \{v\})
        \end{cases}
    \end{align*}
    The matrices $\mathcal{A}, \mathcal{B}$ and the vector $g$ are rationally identifiable by assumption, i.e. they can be constructed for almost all $(\Phi, \Omega)\in \R^\G_\stable$. To construct $f$, note that due to the assumption that $(Y,W,L')$ satisfies the LF-HTC with respect to $v$, it holds that 
    \begin{align*}
    [\sdlint]_{Y, W \cup \{v\}} = [\Fsep]_{L',O}^\top[\sdint]_{L',L'}[\Fsep]_{L',O}^\ast 
    \end{align*}
     The matrix $[\sdlint]_{Y, W \cup\{v\}}$ has at least rank $r$, for generic parameter configurations, but cannot have full column rank (as a matrix over $\R( z)$), since $|L'| = |W| = r$ and condition 2 of the LF-HTC combined with Proposition \ref{prop: invertible subspectrum} applied to the sub-graph $G''= (O \cup L, D_L)$. So, for generic $(\Phi, \Omega)\in \R^\G_\stable$, there is a $f \in \R( z)^W$ such that 
     \begin{align*}
         [\sdlint]_{Y, W}f = [\sdlint]_{Y,v}.
     \end{align*} 
     Equation (\ref{eq: LF-HTC equation}) follows by the same calculation as in \cite{10.1214/22-AOS2221}. Since we have performed the construction over the field of rational functions $\R(z)$, all these relations hold for generic parameter choices, independent of the underlying time series graph. 
     
     It remains to show that for generic parameter choices $(\Phi, \Omega)$, the matrix $\begin{bmatrix}
        \mathcal{A}& \mathcal{B}
    \end{bmatrix}$ is invertible. Let us therefore write $F\coloneqq \det(\begin{bmatrix}
        \mathcal{A}& \mathcal{B}
    \end{bmatrix} )$ which we view as an element in $\R(\Phi, \Omega, z)$. We need to show that $F$ is non-zero. Let $\treks_F$ be the set of trek systems $T'$ such that the following conditions are satisfied:
    \begin{enumerate}
        \item If $\tau \in T'$ is a trek from $y \in Y \cap \htr_{L'}(W \cup \{v\}) $ to $u \in \pa_O(v)$, then $\Left(\tau)$ does not contain any link $x \to y \in D_O$.
        \item If $\tau \in T'$ is a trek from $y \in Y$ to $w \in W$, then $\Right(\tau)$ does not contain any link $v\to w\in D_O$. If in addition $y \in Y \cap \htr_{L'}(W \cup \{v\})$, then it is also required that $\Left(\tau)$ does not contain any link $x \to y\in D_O$.
    \end{enumerate} 
    If $T\in \treks(X,Y)$ such that $G(T)$ is acyclic, then for $(\Phi, \Omega) \in \R^{\G(T)}_\stable$ it holds that
    \begin{align}\label{eq: restriction of F}
        F(\iota_T(\Phi), \Omega,z) &= \sum_{T \in \treks_F \cap \treks_{G(T)}(X,Y)} \sgn(T)\sd^{(T)}(\iota_T(\Phi), \Omega,z),
    \end{align}
    To see that $F$ is not the zero element in $\R(\Phi, \Omega)$ it suffices to identify some $(A, \Omega, z)$ such that $F$ evaluated at that triple is non-zero.
    
    In order to find such a triple, we consider a system of latent factor half-treks $T \in \treks(Y, W \cup \pa_O(v))$ of the form specified in Definition \ref{def: LF-HTC} and that additionally satisfies the properties spelled out in Lemma \ref{lemma: minimal lf half-trek}. Note that the system of latent factor half-treks $T$ is an element in $\treks_F$, which follows from Definition \ref{def: LF-HTC}. Using Lemma \ref{lemma: minimal lf half-trek} we conclude that $\treks_F \cap \treks_{G(T)}(X,Y) = \{T\}$. For $(A, \Omega) \in \R^{\G(T)}_\stable$ equation (\ref{eq: restriction of F}) reduces to
    \begin{align*}
        F(\iota_T(A), \Omega, z) = \sgn(T)\sd^{(T)}(\iota_{T}(A), \Omega, z) = \det([\sd_{\G(T)}(A, \Omega, z)])_{Y, \pa_O(v)\cup W}
    \end{align*}
    and this suffices to conclude that $F$ is non-zero as an element in $\R(\Phi, \Omega, z)$, and therefore that the matrix $\begin{bmatrix}
        \mathcal{A}& \mathcal{B}
    \end{bmatrix}$ is generically invertible.    
\end{proof}
\section{Discussion}
A structural vector autoregressive (SVAR) process is a linear causal model for processes that evolve over a discrete set of time points. Graphically, the time series graph provides information about the detailed time lags with which the processes drive each other. The time series graph is infinite and, depending on the specific lag structure, highly complex. The complexity of the finite process graph, on the other hand, is independent of the lag structure. So, analysing the causal structure in SVAR processes at the level of the process graph is combinatorially a potentially much simpler problem than reasoning about it at the level of the time series graph. The recently established frequency domain representation serves as a causal model for the process graph. The observational information of an SVAR process is represented by the spectral density, i.e. the frequency domain version of the autocovariance. In causal inference, one is interested in recovering the causal model from the observations.
In the context of SVAR processes, this problem reduces to the question of which aspects of the process graph and which link functions can be uniquely identified from the spectral density. In this paper we have given a rigorous algebraic formulation of this question. Using this formulation, we were able to show that $d$- and $t$-separation statements about the process graph are characterised by algebraic conditions on the spectral density. Furthermore, we showed that the latent factor half trek criterion can be applied to the process graph to determine whether certain link functions are generically identifiable in terms of rational operations on the spectral density. 

In relation to future work, it would be interesting to use the spectral characterisation of $d$- and $t$-separation to estimate the process graph from observational data. The aim would be to identify the Markov equivalence class of the process graph directly, rather than identifying the time series graph first. This could be advantageous from a computational point of view, as the number of conditional independence statements to be checked on the process graph is potentially much smaller than on the underlying time series graph. To derive the Markov equivalence class of the process graph, one could either estimate the spectrum as an element in the space of matrices over the rational functions, or one could estimate the evaluations of the spectrum at a finite number of frequencies \cite{liu2010asymptotics, brockwell2009time, rosenblatt1984asymptotic} (using, for example, the discrete-time Fourier transform). Finally, we would need to test whether a rational function is equal to zero or not. If we had an estimator for the entire rational function, it would suffice to test whether all the coefficients defining the numerator of the rational function are zero. On the other hand, we may only have estimates for a certain number of evaluations of a rational function. If the number of points is greater than or equal to the degree of the numerator of the rational function, then we need to test whether the evaluations are equal to zero at the selected frequencies. However, the estimators for these quantities could behave badly from a statistical point of view, so it would be interesting to compare the reduced computational complexity with the potential sacrifice in statistical performance.

It may also be worthwhile to consider the scenario where the noise terms of the SVAR model are non-Gaussian. A random variable that does not follow a Gaussian distribution has non-zero moments not only in degree one (expectation) and two (covariance), but also in the higher order terms. SEM's with non-Gaussian noise terms (i.e. the non-time series case) have been found to exhibit algebraic relationships not only in the covariance, as explained by the trek rule, but also in the higher order terms \cite{robeva2021multi, amendola2023thirdorder}. It is possible to exploit these relationships to extract even more information about the underlying causal model from the observed distribution \cite{tramontano22alearning}. Future work could investigate whether SVAR models with non-Gaussian noise terms also exhibit relationships in the higher order spectra \cite{petropulu1994hos, Nikias1993HigherorderSA, Nikias1993SignalPW} that can be used to recover the frequency domain causal model.
\section*{Acknowledgments}
J.W. and J.R. received funding from the European Research Council (ERC) Starting Grant CausalEarth under the European Union’s Horizon 2020 research and innovation program (Grant Agreement No. 948112).
N.R., A.G. and J.R.  received funding from the European Union’s Horizon 2020 research and innovation programme under Marie Skłodowska-Curie grant agreement No 860100 (IMIRACLI).

N.R. thanks Nils Sturma for encouraging and helpful discussions. 
\appendix
\section{An involution on the rational functions with real coefficients}
\begin{lemma}\label{lemma: multiplicativity conjugate polynomials}
    The map $()^\ast: \R[z] \to \R[z]$ is compatible with  multiplication, i.e. for any two $f, g \in \R[z]$ it holds that 
    \begin{align*}
        (fg)^\ast  &= f^\ast g^\ast  
    \end{align*}
\end{lemma}
\begin{proof}
    Let $f = \sum_{k = 0}^n \alpha_k z^k$ and $g= \sum_{j=0} \beta_j z^j$ be polynomials such that $\alpha_n \neq 0$ and $\beta_m \neq 0 $. Accordingly, the product of $f$ and $g$ is a polynomial of degree $n+m$ and has the following form 
    \begin{align*}
        fg = &\sum_{d= 0}^{n+m} \gamma_d z^d & \gamma_d &= \sum_{l = 0}^d \alpha_{d-l} \beta_l. 
    \end{align*}
    The conjugate of the product, i.e. $(fg)^\ast$ is therefore $\sum_{d = 0}^{n+m} \gamma_{n+m-d} z^d$. On the other hand, we have the polynomial representation of the product of conjugates 
    \begin{align*}
        f^\ast g^\ast &= \sum_{d = 0}^{n+m} \delta_d z^d, & \delta_d &= \sum_{l= 0}^d \alpha_{n-d+l}\beta_{m-l} 
    \end{align*}
    Hence, we need to show that $\gamma_{n+m-d} = \delta_d$ for all $d=0, \dots, n+m$. To see that, note that 
    \begin{align*}
        \gamma_{n+m-d} &= \sum_{l = m-d}^m \alpha_{n + m - d - l} \beta_{l} = \sum_{l' = 0}^d \alpha_{n -d + l'} \beta_{m - l'} = \delta_d , 
    \end{align*}
    where in the first equation we used that $\alpha_k = 0$ for $k > n$ and $\beta_j = 0 > m$, and for the second equation we introduced the index $l' = m - l$. This finshes the proof. 
\end{proof}

\begin{proof}[Proof of Proposition \ref{prop: conjugation of rational functions}]
    First, we need to show that the map $()^\ast$ is well defined, i.e. let $f,g, h \in \R[z]$ be polynomials.
    It follows from Proposition \ref{lemma: multiplicativity conjugate polynomials} that 
    \begin{align*}
        \left(\frac{fh}{gh}\right)^\ast = \frac{(fh)^\ast}{(gh)^\ast} z^{\deg(g) + \deg(h)-\deg(f) - \deg(h)} = \frac{f^\ast h^\ast}{g^\ast h^\ast} z^{\deg(g)- \deg(f)} = \left(\frac{f}{g}\right)^\ast
    \end{align*}
    This shows that the conjugate of a rational function does not depend on the choice of representative.
    
    Next, we need to show that conjugation commutes with multiplication in $\R(z)$. Let us therefore fix four polynomials $f_{i}, g_{i} \in \R[z]$ where $i=1,2$. It then holds by Proposition \ref{lemma: multiplicativity conjugate polynomials} that 
    \begin{align*}
        \left( \frac{f_1f_2}{g_1g_2} \right) ^\ast &= \frac{(f_1f_2)^\ast}{(g_1g_2)^\ast} z^{\deg(g_1) + \deg(g_2) - \deg(f_1)- \deg(f_2)} = \left( \frac{f_1}{g_1} \right)^\ast \left( \frac{f_2}{g_2} \right)^\ast 
    \end{align*}

    To show that conjugation of rational functions is linear, i.e.
    \begin{align*}
        \left(\frac{f_1}{g_1} + \frac{f_2}{g_2}\right)^\ast &=  \left( \frac{f_1}{g_1} \right)^\ast + \left( \frac{f_2}{g_2} \right)^\ast
    \end{align*}
    we begin with the case where $g_1=g_2=1 \in \R[z]$, i.e. the constant monomial. Furthermore, we set $d_1 = \deg (f_1)$ and $d_2= \deg(f_2)$ and w.l.o.g. we assume that $d_1 \geq d_2$. Let
    \begin{align*}
        f_1 &= \sum_{k=0}^{d_1} \alpha_k z^k & f_2 &= \sum_{j=0}^{d_2} \beta_j z^j,
    \end{align*} and we assume that $d_1 > d_2$ so that $\deg(f_1 + f_2) = d_1 $. In this case it holds that 
    \begin{align*}
        \left( \frac{f_1 + f_2}{1}\right)^\ast &= \frac{\sum_{k=0}^{d_1} \alpha_{d_1 - k}z^k + \sum_{j = 0}^{d_1} \beta_{d_1 - j} z^{j}}{z^{d_1}} \\
        &= \left( \frac{f_1}{1} \right)^\ast  + \frac{ z^{d_1 -d_2}\sum_{j' = 0}^{d_2} \beta_{d_2 - j'} z^{j'} }{z^{d_1}} \\
        &= \left( \frac{f_1}{1} \right)^\ast + \left( \frac{f_2}{1} \right)^\ast, 
    \end{align*}
    where we set $\beta_j =0 $ for $j > d_2$, and in the third equation we used the index translation $j = d_1 - d_2 + j' $. In the second case ($d=d_1=d_2$), one has to consider the situation in which the degree of the sum decreases, i.e. $\deg(f_1 + f_2) < d_1$. This happens when some of the coefficients of $f_1$ and $f_2$ offset each other. So, let $\gamma_l = \alpha_l + \beta_l$ and suppose that $\gamma_l =0$ for all $d' < l \leq d$, i.e. $\deg(f_1+f_2) =d' < d$. Then we compute as follows 
    \begin{align*}
        \left( \frac{f_1}{1}\right)^\ast + \left( \frac{f_2}{1} \right)^\ast  &
        = \frac{\sum_{k=0}^d \alpha_{d-k}z^k}{1} + \frac{\sum_{j=0}^d \beta_{d-j}z^j}{1} \\ &
        = \frac{\sum_{k=0}^d \gamma_{d-l}z^l}{z^d} \\ &
        = \frac{z^{d-d'} \sum_{m= 0}^{d'} \gamma_{d' - m }z^m}{z^d} \\ &
        = \left( \frac{f_1 + f_2}{1} \right)^\ast ,
    \end{align*}
    where we used the index translation $l = d-d' +m$. 

    We are eventually prepared to conclude with the linearity of the conjugation of rational functions. Precisely,
    \begin{align*}
        \left( \frac{f_1}{g_1} + \frac{f_2}{g_2} \right)^\ast &=  \left( \frac{f_1 g_2 + f_2g_1}{g_1g_2} \right)^\ast \\
        &= \left( \frac{f_1 g_2 + f_2g_1}{1} \frac{1}{g_1g_2} \right)^\ast \\
        &= \left(\left(\frac{f_1g_2}{1}\right)^\ast + \left(\frac{f_2g_1}{1}\right)^\ast \right)  \left(\frac{1}{g_1g_2} \right)^\ast \\
        &=
        \left( \frac{f_1}{g_1}\right)^\ast + \left( \frac{f_2}{g_2} \right)^\ast.
    \end{align*}
    For the third equation we used that conjugation and multiplication of rational functions commute and that conjugation respects addition on the subring $\R[z] \subset \R(z)$. For the last equation we used again that conjugation and multiplication of rational functions commute. 

    Lastly, we check that conjugation is the inverse to itself. Let $f = \sum_{k=0}^n\alpha_k z^k$ and $g = \sum_{j=0}^m \beta_j z^j$ be two polynomials in $\R[z]$. Then we compute 
    \begin{align*}
        \left( \left( \frac{f}{g} \right)^\ast \right)^\ast &= \left( \frac{z^m \sum_{k=0}^n \alpha_{n-k}z^k}{z^n \sum_{j=0}^m \beta_{m-j}z^j} \right)^\ast \\
        &= \left( \frac{\sum_{k=0}^{n+m} \hat{\alpha}_{n+m -k} z^k}{\sum_{j=0}^{n+m} \hat{\beta}_{n+m -j} z^j} \right)^\ast \\
        &= \frac{\sum_{k=0}^{n+m} \hat{\alpha}_{k} z^k}{\sum_{j=0}^{n+m} \hat{\beta}_{j} z^j} \\
        &= \frac{f}{g}
    \end{align*}
    In the lines above we used the augmented coefficient vectors $\hat{\alpha}, \hat{\beta}\in \R^{n+m}$, which are defined as follows
    \begin{align*}
        \hat{\alpha}_k &= \begin{cases}
            \alpha_k & \text{ if } 0 \leq k \leq n \\
            0 & \text{ if } n < k \leq n+m 
        \end{cases} & 
        \hat{\beta}_j &= \begin{cases}
            \beta_j & \text{ if } 0\leq j \leq m \\
            0 & \text{ if } m < j \leq n+m
        \end{cases}
    \end{align*}
    The last claim follows as $z^\ast = z^{-1}$ for $z \in S^1$ and direct inspection.
\end{proof} 

\section{Proof of the spectral Gessel-Viennot Lemma}
\begin{proof}[Proof of Proposition \ref{prop: spectral gessel viennot lemma}]
    We show this by induction on the cardinality of $X$ resp. $Y$. Let $l=1$. Then the statement follows from the spectral path rule \cite{reiter2023formalising}. 
    Assume the statement holds for $|X| = l-1$. We now wish to show that this implies that the statement also holds for cardinality equal to $l$, 
    i.e., $X = \{x_1, \dots, x_l \}$ and $B = \{y_1, \dots, y_l \}$. In order to avoid overly complicated notation we write $\Fsep^\infty$ to abbreviate $(I - \Fsep)^{-1}$. We employ the Leibnitz expansion of the determinant 
    \begin{align*}
        \det([\Fsep^\infty]_{X,Y}) &= \sum_{j = 1}^l \fsep^{\infty}_{x_1, y_j} (-1)^{1+j} \mathfrak{m}_{1,j},
    \end{align*}
    where $\mathfrak{m}_{i,j} = \det ([\Fsep^{\infty}]_{X_i,Y_j})$ and $X_i = X \setminus \{x_i\}$ resp $Y_j = Y \setminus \{y_j \}$.  
    Applying the induction hypothesis we get that 
    \begin{align*}
        \mathfrak{m}_{1, j} &= \sum_{\Pi \in \mathbf{N}_0(X_1, Y_j)} \sgn(\Pi) \prod_{\pi \in \Pi} \fsep^{(\pi)}
    \end{align*}
    On the other hand we know from the spectral path rule that 
    \begin{align*}
        \mathfrak{h}^\infty_{x_1,y_j} &= \sum_{\pi \in \mathbf{P}_0(x_1, y_j)} \fsep^{(\pi)}.
    \end{align*}
    Let us fix a path $\pi \in \mathbf{P}_0(x_1, y_j)$ and a system of paths $\Pi \in \mathbf{P}_0(X_1, Y_j)$. Hence, $\Pi \cup \{ \pi \}$ constitutes a system of directed paths from $X$ to $Y$. 

    The system $\Pi \cup \pi$ induces a permutation on the set of $l$ elements. We now show that $\sgn(\Pi \cup \pi) = (-1)^{1+j} \sgn(\Pi)$. This can be seen using the description of the sign of permutation as the determinant of its corresponding permutation matrix. We can therefore conclude with 
    \begin{align}
        \det([\Fsep^\infty]_{X,Y}) &= \sum_{j = 1}^l \sum_{\pi \in \mathbf{P}_0(x_1, y_j)} \sum_{\Pi \in \mathbf{N}_0(X_1, Y_j)} \sgn(\Pi \cup  \pi) \fsep^{(\pi)} \Fsep^{(\Pi)} 
    \end{align}
    We now assume that $\pi$ has an intersection with at least one of the paths in $\Pi$. We can then find a path $\pi' \in \Pi$ and nodye $v \in V$ such that $\pi$ and $\pi'$ intersect in $v$, and such that $v$ is the first node visited by $\pi$ that is also visited by some path in $\Pi$. Since $\Pi$ is non intersecting it follows that the node $v$ and the path $\pi'$ are unique. The directed path $\pi'$ connects some $x_i \in X_1$ to some $y_k\in Y_j$. Then we can write the paths $\pi$ and $\pi'$ as the concatenations
    \begin{align*}
        \pi &= \pi_{x_1, v} \oplus \pi_{v, y_j} & \pi' &= \pi_{x_i, v}' \oplus \pi_{v, y_k}',
    \end{align*}
    We use the symbol $\oplus$ to denote concatenation of paths. Using this concatenation we proceed by constructing two new paths 
    \begin{align*}
        \hat{\pi} &= \pi_{x_1,v} \oplus \pi'_{v, y_k} & \tilde{\pi} &= \pi_{x_i, v}' \oplus \pi_{v, y_j}.
    \end{align*} 
    Then $\Pi' = \Pi \setminus \{\pi' \} \cup \{\tilde{\pi} \}\in \mathbf{P}(X_1, Y_k)$ is a system of non intersecting directed paths from $X_{1}$ to $Y_{k}$. That the system is non intersecting follows from the choice of $v$ and $\pi'$. Furthermore, $\hat{\pi}$ is a directed path from $x_1$ to $y_j$. It follows that 
    \begin{align*}
        \fsep^{(\pi)} \Fsep^{(\Pi)} &= \fsep^{(\hat{\pi})} \Fsep^{(\Pi')}
    \end{align*}
    Let us denote by $\gamma$ the permutation that is associated with $\Pi \cup \pi$, and by $\hat{\gamma}$ the permutation that is associated with $\Pi' \cup \hat{\pi}$. It holds that $\sgn(\gamma) = - \sgn(\hat{\gamma})$. This is because $\hat{\gamma}$ can be obtained from $\gamma$ by post-composing with a transposition. Specifically, we define $\tau$ to be the following transposition 
    \begin{align*}
        \tau(y_j) &= y_k & \tau(y_k) &= y_j.
    \end{align*}
    Using this transposition we see that $\hat{\gamma} = \tau \circ \gamma $, which combined with the multiplicativity of the sign yields $\sgn(\hat{\gamma}) = - \sgn(\gamma)$. 

    It remains to show that $\det([I - \Fsep]_{X,Y}^{-1})$ is non-zero if and only if there there is at least one system of non-intersecting paths from $X$ to $Y$. First, if there is no such system then it follows by the identity we just showed that the determinant must be zero. So suppose $ \mathbf{P}(X,Y) = \{\Pi_{1}, \dots, \Pi_l\}$. We now wish to show that the rational function 
    \begin{align}\label{eq: rational function path systems}
        \Fsep^{(\Pi_1)} + \cdots + \Fsep^{(\Pi_l)} \in \R(\Phi, z) \setminus \{0 \},
    \end{align}
    i.e. that this is a non-zero rational function in the indeterminates $\Phi$ and $z$. In order to show this claim we introduce a few polynomials. Let $\pi$ be a directed path, then 
    \begin{align*}
        p^{(\pi)} &= \prod_{v \to w \in \pi} \varphi_{v,w} & q^{(\pi)} &= \prod_{v \to w \in \pi} 1 - \varphi_{w,w},
    \end{align*}
    of which the first is the numerator and the second is denominator of the path function $\fsep^{(\pi)}$. 
    These are polynomials in $\R[\Phi, z]$ and note that the coefficient of the degree zero monomial of $q^{(\pi)}$ is equal to one. Now suppose that $\pi \in \mathbf{P}_0(x,y)$ is a directed path that visits every vertex at most once and $\Pi \in \mathbf{P}(X,Y)$ a system of directed paths, then we introduce the following polynomials 
    \begin{align*}
          p^{(\Pi)} &= \prod_{\pi \in \Pi} p^{(\pi)} & q^{(\Pi)} &= \prod_{\pi \in \Pi} q^{(\pi)}.
    \end{align*}
    Again, we read all of these polynomials as elements in the polynomial ring $\R[\Phi, z]$. The rational function (\ref{eq: rational function path systems}) is non-zero if and only if the polynomial 
    \begin{align*}
        p^{(\Pi_1)}Q_1 + \dots +p^{(\Pi_l)}Q_l\in \R[\Phi, z], & & Q_i &= \prod_{j \neq i}^l q^{(\Pi_j)} \in \R[\Phi, z] 
    \end{align*}
    is non-zero. We show this by identifying a monomial $x$ that has non-zero coefficient in $p^{\mathrm{P}(\Pi_1)}Q_1$, but the coefficient of the monomial $x$ in the polynomial $p^{\mathrm{P}(\Pi_j )}Q_j$ is zero for all $j \geq 2$. Let us define the monomial $x$ as follows 
    \begin{align*}
        x &= \prod_{\pi \in \Pi_1} \prod_{v\to w \in \pi} \phi_{v,w}(m_{v,w})z^{m_{v,w}}, & m_{v,w} &= \min\{k \mid v \to_k w \}.
    \end{align*}
    Note that the coefficient of $x$ in $p^{(\Pi_1)}Q_1$ is equal to one, which can be seen from the fact that the coefficient of the monomial with degree zero is equal to one in every $Q_i$. It remains to show that the coefficient of $x$ in $p^{(\Pi_j)}Q_j$ is zero for all other $j$. Let us therefore consider the sets of links $D_i = \{ v\to w \mid v \to w \in \pi, \pi \in \Pi_i \} $. Note that for any two distinct $i,j$ the assumption $\Pi_i \neq \Pi_j$ are systems of non-intersecting paths implies that also the corresponding sets of edges are different, i.e. $D_i \neq D_j$. This implies that the monomial $x$ has a non-zero coefficient only in the polynomial $p^{(\Pi_1)}Q_1$. This allows us to conclude that the rational function (\ref{eq: rational function path systems}) is indeed non-zero. 
\end{proof}

\section{Auxiliary Lemma for proving the LF-HTC identifiability result}

Suppose $\G$ is a time series graph with exogenous latent structure. In the main part of the paper we introduced $\R^\D_\stable$ as the space of SVAR parameters that are compatible with $\G$ and give rise to process which can be represented as an SEP in the frequency domain. The following Lemma is used in the main part of the paper to prove the spectral version of the latent factor half-trek identifiability result. The Lemma is closely related to Lemma 1 in the supplementrary material of \cite{foygel2012half}.

\begin{lemma}\label{lemma: minimal lf half-trek}
    Suppose $X,Y \subset O $ are disjoint subsets and let $T' =(\tau_1, \dots, \tau_n ) \in \treks(X,Y)$ be a system of latent factor half-treks without sided intersection. Then there exists a sub-system $T=(\tau_1, \dots, \tau_n) \in \treks(X,Y)$ of latent factor half-treks together with an ordering of the elements in $X$ and $Y$ and $T$ such that: 
    \begin{enumerate}
        \item $ \tau_i \in \treks(x_i, y_i)$ a latent factor half-trek such that $\tau_i$ goes through $x_j$ only if $ i \geq j$ and $x_i$ appears on $\tau_i$ exactly once 
        \item the associated subgraph $G(T)$ is acyclic.
    \end{enumerate}
    For such $T$ and associated parameter $\Phi \in \R^{\D(T)}_\stable$ it holds that 
    \begin{align}\label{eq: spectrum minimal trek}
        \det([\sd_{\G(T)}(\Phi, \Omega)]_{X,Y}) = \sd^{(T)}(\Phi, \Omega)
    \end{align}   
\end{lemma}
\begin{proof}
    Using the same procedure as in the proof of Lemma 1 in the supplementary material of \cite{foygel2012half}, one can reduce the system $T'$ to a system of latent factor half-treks $T\in \treks(X,Y)$ with property 2.) and furthermore reorder the elements in $X$, $Y$ and $T$ so that 1.) holds. 

    To show the equation (\ref{eq: spectrum minimal trek}) we need to show that $\treks_{G(T)}(X,Y) = \{ T\}$. Let $T' = (\tau_1', \dots, \tau_n')$ be a trek system without sided intersection from $X$ to $Y$ such that $\tau'_i \in \treks(x_i, y_{\alpha(i)})$, where $\alpha$ is a permutation on $\{1, \dots, n \}$. We need to show that $T' = T$. 

    Each $\tau_k$ is a latent factor half-trek, so it has one of the following two forms 
    \begin{align*}
        x_k \leftarrow l_k \to z_1^{k} \to \cdots z^k_{I_k} =y_k && x_k \to z_1^{k} \to \cdots z^k_{I_k} =y_k,
    \end{align*}
    where $l_k \in L$ and $z_j^k \in O$. Each $x\in X$ has at most one outgoing link on $G(T)$, this follows from the fact that $T$ has no side intersection and $G(T)$ is acyclic.
    
    We start by showing that $\tau_n' = \tau_n$. 
    To construct a trek $\tau'_n$ starting at $x_n$ we need an edge in $D(T)$ that either points to $x_n$ or emerges from $x_n$. By assumption $x_n$ appears only on $\tau_n$ and on $\tau_n$ the vertex $x_n$ appears exactly once. So there is exactly one edge involving $x_n$, which means that the first link on $\tau_n'$ must be the first link on $\tau_n$. 
    
    If the first link is $x_n \leftarrow l_n$, then the second link on $\tau_n'$ cannot be $l_n \to x_n$ again, because $X$ and $Y$ are disjoint and there is no path from $x_n $ to any $y\in Y$ on $G(T)$. Since $l_n$ has exactly two emerging links on $G(T)$, the second link on $\tau_n'$ must be $l_n \to z_1$. In particular, if $z_1 \in Y$, then $z_1 = y_n$. If $z_1 \neq y_n$, then there can only be one edge coming out of $z_1$ on $G(T)$, since $T$ has no sided intersection. So the next edge on $\tau_n'$ is $z_1 \to z_2$. This reasoning can be repeated until $y_n$ is reached. The case where $\tau_n = x_n \to z_1 \to \cdots$ works analogously. So, we conclude that $\tau_n' = \tau_n$.  
    
    Now suppose we have shown that $\tau_k = \tau_k'$ for all $m < k \leq n$. We now want to show that $\tau_m' = \tau_m$. By the assumption on $T$, the node $x_m$ can only appear on $\tau_j$ with $j \geq m$. To construct a trek starting at $x_m$ we need either a link pointing to $x_m$ or a link coming from $x_m$. 

    \textsc{Case 1:} If $x_m$ only appears on $\tau_m$, then we can repeat the same argument as above to show that $\tau_m' = \tau_m$.

    \textsc{Case 2:} The node $x_m$ also appears on some $\tau_l = \tau_l'$ with $l > m$, i.e. $x_m = z^l_j$ for a $j \in \{2, \dots, I_l -1 \}$, where $l$ and $j$ are unique, since $T$ is without sided intersection and $G(T)$ is acyclic. In this case there can be only one edge on $G(T)$ starting from $x_m$, i.e. $x_m = z^l_j \to z^l_{j+1}$. By assumption, this edge already occurs at $\tau_l'= \tau_l$. So the first edge on the path $\tau_l'$ must point to $x_m$ (otherwise $\tau_m'$ would have a sided intersection with $\tau_l'$). There are now exactly two links we can choose to start $\tau'_m$, i.e. $x_m \leftarrow l_m$ or $z \to z_j^l$, where $z$ is either $z_{j-1}^l$ or $l_l$ or $x_l$. We now need to show that the first link on $\tau'_m$ is indeed $x_m \leftarrow l_m$. Once we have shown that, we can proceed as above to prove that $\tau_m' = \tau_m$, since $\tau_m'$ cannot be of the form $x_m \leftarrow l_m \to x_m$ (then $\tau_m'$ would intersect with $\tau_l'$).

    Suppose $j > 1$ and also that there are some $1 \leq i < j$ such that $z_i^l = x_q \in X$. Otherwise, a trek $\tau'_m$ where the first link was $x_m \leftarrow z$ could not reach a point in $Y$, contradicting $T' \in \treks(X,Y)$. 

    So let $i^\ast = \min\{ i \mid i < j, z^l_i \in X\}$ and $z_{i^\ast}^l = x_q$, so the first link on $\tau_q$ is $x_q \leftarrow l_q$. In this case $\tau_m' = (x_m \leftarrow \cdots \leftarrow x_q \leftarrow \cdots)$. Given that $\tau_m'$ ends in $Y$, there are two options for $\tau'_m$, namely
   \begin{align}
        x_m \leftarrow \cdots x_q \leftarrow l_q \to x_q \label{option 1} \\
        x_m \leftarrow \cdots x_q \leftarrow l_q \to z_1^q \label{option 2} 
   \end{align}
   The trek $\tau_m'$ cannot be of the form (\ref{option 2}), since this would result in an intersection with $\tau_l$. Nor can $\tau_m'$ be of the form (\ref{option 1}), because that would mean that either $\tau_m'$ or $\tau_l'$ would have a sided intersection with the trek $\tau_q'$ (i.e. the trek in $T'$ starting at $x_q$). Since this argument can be repeated for every $i<j$ such that $z_i^l \in X$, we conclude that the first link on $\tau_m'$ is indeed $x_m \leftarrow l_m$. As mentioned above, this implies that $\tau_m' = \tau_m$. 

    This concludes the proof for $\treks(X,Y)= \{T \}$ implying equation (\ref{eq: spectrum minimal trek}). 
\end{proof}

\bibliographystyle{plain}
\bibliography{paper-ref}
\end{document}